\documentclass[10pt]{amsart}
\usepackage{amsmath, palatino, mathpazo, amsfonts, amssymb, mathtools}
\usepackage[mathscr]{eucal}
\usepackage[all]{xy}
\usepackage{datetime}
\newtheorem{theorem}{Theorem}[section]
\newtheorem{lemma}[theorem]{Lemma}
\newtheorem{proposition}[theorem]{Proposition}
\newtheorem{corollary}[theorem]{Corollary}
\theoremstyle{remark}
\newtheorem{remark}[theorem]{Remark}

\newtheorem{definition}[theorem]{Definition}

\numberwithin{equation}{section}

\newcommand{\Mbar}{\overline{\mathcal{M}}}

\newcommand{\op}{\operatorname}
\newcommand{\llangle}{\langle \! \langle}
\newcommand{\rrangle}{\rangle \! \rangle}
\newcommand{\A}{\mathscr{A}}

\begin{document}
	
\title[Virasoro constraints with weighted pointed curves]{Virasoro constraints for moduli of weighted pointed stable curves}% and maps}
	
	\author[Y.-C~Chou.]{You-Cheng Chou}
	\email{chou@math.utah.edu}
	
	\author[Y.-P.~Lee]{Yuan-Pin~Lee}
	\email{yplee.math@gmail.com, yplee@math.utah.edu}
	
\address{Institute of Mathematics, Academia Sinica, Taipei 10617, Taiwan, and
Department of Mathematics, University of Utah, 	Salt Lake City, Utah 84112-0090, U.S.A.}
	
	\date{\today}
	%%%%%%%%%%%%%%%%%%%%%%%%%%%%%%%%%%%%%%%
	
	\maketitle

	%\small
	%\tableofcontents
	%\normalsize

\begin{abstract}
We formulate Virasoro constraints for the generating functions of the intersection numbers on Hassett's moduli of weighted pointed curves and show that they are governed by the KdV integrable hierarchy. %albeit in a subtle way.
\end{abstract}

\section{Introduction}

In this paper, we extend the usual Virasoro constraints for the Witten--Kontsevich theory on the moduli of pointed stable curves  
to the case where the marked points are weighted, in the sense of B.~Hassett \cite{bH}.

The geometry linking the intersection numbers of $\psi$-classes on the Hassett and Deligne--Mumford moduli of pointed curves is well known.
In particular, there are natural morphisms from the moduli of Deligne--Mumford stable curves to moduli of weighted pointed stable curves,
and the pullbacks of the $\psi$-classes of the latter can be written as the corresponding $\psi$-classes on the former plus boundaries.
We refer the readers to the recent preprint by V.~Blankers and R.~Cavalieri \cite{BC} for an excellent review of background material.
Among other things, \cite{BC} formulates the underlying combinatorics in a nice way.
For the purpose of this paper their results are not needed.
Rather, we rework the underlying combinatorics in Section~\ref{s:1} in order to formulate our main combinatorial gadget: the $h$ function in Section~\ref{s:2}.

The main results of the paper are contained in Section~\ref{s:2}.
Let $\A$ be an additively closed set of weights, as defined in Section~\ref{s:1.2}.
One can define %the phase space $\mathbf{H}^{\A}$ and 
the generating function $F^{\A} (t)$ of all intersection numbers of $\psi$-classes on the weighted pointed moduli of curves,  defined in \eqref{e:2}.
We show that for any weight $a \in \A$, there is a set of partial differential operators $V_a: = \{ L_{k;a}^{\A} \}_{k \geq -1}$ which satisfy the (centerless) Virasoro relations in Corollary~\ref{c:2.9} and annihilate the generating function 
\[
 L_{k;a}^{\A} (e^{F^{\A}}) =0, \quad \forall a \in \A, \, k \geq -1
\]
in Theorem~\ref{t:2.13}.
Furthermore, the Virasoro constraints uniquely determine $F^{\A}$ up to initial conditions
\[
\langle \tau_{1;a} \rangle = \frac{1}{24}, \quad \langle \tau_{0;a_1}\tau_{0;a_2}\tau_{0;a_3} \rangle = 1
\quad\mbox{ for any $a, a_1, a_2, a_3 \in \A$ }.
\]
These are explained in Section~\ref{s:2}, with certain technical combinatorial details verified in Section~\ref{s:3}.
The main point of the proofs is a systematic reduction to the weight $1$ case by combinatorics.

The vector space $V_{\A}$ generated by operators $ \{ L_{k;a}^{\A} \}_{k \geq -1, a\in\A}$ is shown to be a semi-direct product of the usual (centerless) Virasoro Lie algebra with an abelian Lie algebra in Section~\ref{s:2.4}.
Results from representation theory \emph{imply} that the ``nice representations'' (e.g., highest weight representations) of such Lie algebras are all induced from those of the Virasoro algebra \cite{CK}.
Given the relationship between representations of the Virasoro algebra and integrable systems, this in turn \emph{suggests} that the corresponding integrable system are the same KdV hierarchy  for any $\A$.
In Section~\ref{s:KdV}, we show this is indeed the case by an explicit change of variables.
Furthermore, the explicit formulae of the infinite commuting flows and the initial condition of the integrable hierarchy are found.
Extensions to this work to include target spaces as well as quasimap spaces will be discussed in a subsequent work.

\subsection*{Acknowledgements}
We would like to thank Renzo Cavalieri, Shun-Jen Cheng,  Honglu Fan, and Youjin Zhang for helpful conversations and correspondences.
Cavalieri read an early draft and gave many valuable comments.
Part of the work was done during the first author's visit to Academia Sinica.
Both authors are supported in part by the Academia Sinica and the %National Science Foundation, while
the second author is also supported in part by the Simons Foundation.

\section{Moduli of weighted pointed curves and correlators} \label{s:1}
%\begin{itemize}
%\item 

\subsection{Correlators with weighted points}
Let $\overline{ \mathcal{M} }_{ g , a_1, \ldots,  a_n }$ be the moduli space of stable genus $g$ curves, with ordered $n$ marked points with weights 
\[
 \underline{a} = (a_1, a_2, \dots, a_n)
\] 
with $a_i \in [0^+, 1]$ \cite{bH}.
The ``unweighted'' moduli space $\overline{ \mathcal{M} }_{g,n}$ in the sense of Deligne--Mumford is denoted here as $\overline{ \mathcal{M} }_{g, 1, 1, \dots, 1 }$. 
We note that the (infinitesimal) weight $0^+$ is allowed with the standard arithmetic properties:
\begin{itemize}
\item $\forall a \in (0,1], \, 0^+ < a$.
\item Let $a^+ := a + 0^+$, $0^+ + 0^+ = 0^+,  a + b^+ = a^+ + b^+ = (a+b)^+$.
\end{itemize}
By the result of B.~Hassett \cite{bH}, $\Mbar_{ g , a_1, \ldots , a_n }$ are smooth, irreducible Deligne--Mumford stacks;
there are canonical surjective birational morphism $\pi: \Mbar_{g,n} \to  \Mbar_{ g , a_1,\ldots, a_n }$.

Per usual, we define the correlators with weighted points as
\[
\langle \tau_{\underline{k}; \, \underline{a}} \rangle_g := \langle \tau_{ k_1 ; a_1 } \dots \tau_{ k_n ; a_n } \rangle_{g}  :=
\int_{\overline{ \mathcal{M} }_{ g , a_1, \dots, a_n } } \psi_1^{k_1} \cdots \psi_n^{k_n}.
\]
For dimensional reasons, the above correlators vanish unless 
\begin{equation} \label{e:g}
 3g -3 +n = \sum_{i=1}^n k_i.
\end{equation}
When $g$ is omitted from the notation, it is determined by the above equation uniquely.
(The correlators vanish if $g$ is fractional.)
The following convention is adapted throughout the paper
\[
 \tau_k := \tau_{k;1}.
\]
The correlators with arbitrary weights can be related to those with ``normal'' weights $a_i=1, \, \forall i$.
Let us introduce some notation.
Let %$\underline{k} \in \mathbb{Z}_{\geq 0}^n$ and 
$\underline{a} =(a_1, a_2, \ldots, a_n)$ be an ordered list of weights.
A \emph{partition} of the components of $\underline{a}$ is a partition of the set $\{ a_1, \ldots, a_n \}$, i.e.,
\[
 p(\underline{a}) = (a_{11}  a_{12} \ldots a_{1l_1}) (a_{21} a_{22} \ldots a_{2 l_2} ) \ldots ( a_{\rho 1} \ldots a_{\rho l_{\rho}} ),
\]
 such that 
\[
 n= \sum_{i=1}^{\rho} l_i , \qquad  
  \{ a_{11} , a_{12}, \ldots , a_{\rho 1} , \ldots , a_{\rho l_{\rho}}  \} = \{ a_1, \ldots, a_n \}
\] 
and is subject to the following equivalence relations:
\begin{itemize}
\item permutations of $a_{ij}$ within the $i$-th part (parenthesis) are considered the same element;
\item permutations of the $i$-th and the $j$-th parts are considered the same element.
\end{itemize}
A partition is called \emph{admissible} if for \emph{all} parts 
\[
 \sum_{j=1}^{l_i} a_{i j} \leq 1, \quad \forall i. 
\]
%(It is called totally unstable partitions in \cite{BC}.)

We introduce the following notation.
Let $\underline{k}$ and $ \underline{a}$ be two $n$-tuples.
An \emph{admissible partition} $p (\underline{k}; \underline{a})$ is an admissible partition of $\underline{a}$ such that
$\underline{k}$ just ``goes along for the ride''. That is, 
\[
 p (\underline{k}; \underline{a}) =  ( [ k_{11} a_{11} ]   \ldots [ k_{1 l_1}  a_{1 l_1}] ) \ldots ( [k_{\rho 1} a_{\rho 1} ] \ldots  [ k_{\rho l_{\rho}} a_{\rho l_{\rho}} ] )
\]
such that, $\forall i , j$, $[k_{ij} a_{ij}] = [k_{i'} a_{i'} ]$ for some $i'$ (the same index). 
Denote by $\mathcal{P} ( \underline{k}; \underline{a})$ the \emph{set of all admissible partitions} $p (\underline{k}; \underline{a})$.
Define
\[
 \langle \tau_{p(\underline{k}; \, \underline{a})} \rangle := \langle \tau_{p(\underline{k} ) } \rangle := \langle \tau_{ 1+ \sum_{j=1}^{l_1} (k_{1j} -1) } \ldots \tau_{ 1+ \sum_{j=1}^{l_{\rho}} (k_{\rho j} -1)  } \rangle .
\]
We note that on the right hand side of the above equation all marked points have \emph{weights $1$}.

For example, let $n=8$. Let $p(\underline{a}) = (a_1 a_7) ( a_2 a_3 a_6) ( a_4 a_5 a_8)$ be an admissible partition and 
let $p (\underline{k}; \underline{a} )$ the the corresponding admissible parition of the pair. Then
\[
 \langle   \tau_{ p ( \underline{k}; \, \underline{e} ) }  \rangle =
   \langle  \tau_{k_1 + k_7-1} \tau_{k_2 + k_3 + k_6 - 2 } \tau_{k_4 + k_5 + k_8 - 2} \rangle.
\]

Let $p \in \mathcal{P}(\underline{k}; \underline{a})$, and let $\rho (p)$ be the length (number of parts) of $p$. Define 
\[
\op{codim}(p) := n- \rho( p ).
\]
The following proposition, first apeared in \cite{AG}, relates the correlators with arbitrary weights to those with weights one.
Indeed, it was this proposition which motivates the above definitions.

\begin{proposition}[\cite{AG}] \label{p:1}
All weighted correlators can be reconstructed from the unweighted ones. More precisely,
\begin{equation} \label{e:1}
\langle \tau_{\underline{k}; \, \underline{a}} \rangle 
= \sum_{  p \in \mathcal{P} ( \underline{k}; \underline{a} )  } 
(-1)^{ \op{codim} (p) } \langle  \tau_{ p ( \underline{k} ) } \rangle 
\end{equation}
where the RHS is a sum of correlators of weight $1$.
We note that by convention $\langle  \tau_{l_1} \dots \tau_{l_n} \rangle = 0$, if $l_i < 0$ for some $i$.
\end{proposition}

Since our formulation is slightly different from that of \cite{AG} and we need variations of results not strictly contained there,
a proof of this proposition will be given here, after the following lemmas.

\begin{lemma} \label{l:1}
\[
\sum_{i=1}^{M} \binom{M}{i} \binom{i-1}{n} (-1)^i = (-1)^{n+1}, \mbox{ for $n < M$ }.
\]
Here we use the convention: $\binom{m}{n}=0$ if $n>m$.
\end{lemma}

\begin{proof}
The above identity can be proved by the following two observations.
\begin{enumerate}
\item For $n=0$, we have
\[
\sum_{i=1}^M \binom{M}{i} (-1)^i = (1-1)^M -1 = -1
\]
\item For $M>n>0$, we have
\[
\sum_{i=1}^M \binom{M}{i} 
\left( 
\binom{i-1}{n} + \binom{i-1}{n-1} 
\right) (-1)^i
=
\sum_{i=1}^M \binom{M}{i} \binom{i}{n} (-1)^i =0,
\]
since it can be understood as the coefficient of $x^n$ of $( 1 - (1+x) )^M$.
\end{enumerate}
The proof of the Lemma follows from induction on $n$ by (1) and (2).
\end{proof}

\begin{lemma} \label{l:2}
Assume $| \underline{a}| := \sum_i a_i =1$, then
\[
 \langle \tau_{e_1;a_1} \dots \tau_{e_n;a_n} \rangle
=  \langle \tau_{e_1;a_1} \dots \tau_{e_{n-1}; a_{n-1}} \tau_{e_n; a_n^+} \rangle
+ (-1)^{n-1} \langle \tau_{ e_1+\dots + e_n -n+1 } \rangle .
\]

\end{lemma}

\begin{proof}
Consider the birational morphism
\[
\pi: \overline{ \mathcal{M} }_{g, a_1,\dots, a_n^+}  \rightarrow \overline{ \mathcal{M} }_{g, a_1,\dots, a_n}.
\]
It is easy to see that 
\[
\psi_i = \pi^* \psi_i + \Delta
\]
where $\Delta$ is the divisor on $\Mbar_{g,a_1,\dots ,a_n^+}$ whose general elements consist curves with a single node separating a rational and a genus $g$ components, such that all $n$ marked points lie on the rational component.
Basic facts on deformations of nodal curves imply that
\[
 \Delta \cdot \Delta = (-\psi' -\psi'') \Delta, 
\]
where $\psi'$ and $\psi''$ are the $\psi$-class of the nodal point on the rational and genus $g$ components respectively.
We have
\[
\begin{split}
&\langle \tau_{e_1;a_1} \dots \tau_{e_{n-1}; a_{n-1}} \tau_{e_n; a_n^+} \rangle \\
= & \int_{  \overline{ \mathcal{M} }_{g, a_1,\dots, a_n^+}  }
( \pi^*\psi_1 + \Delta )^{e_1} \dots  ( \pi^* \psi_{n} + \Delta )^{e_n}
\\
= &
\langle \tau_{e_1;a_1} \dots \tau_{e_{n-1}; a_{n-1}} \tau_{e_n; a_n} \rangle \\
& \quad + \sum_{i=1}^{ |\underline{e}|} \binom{|\underline{e}|}{i} \binom{i-1}{n-2} (-1)^{i-1}
\int_{\overline{\mathcal{M}}_{0,n+1} } (\psi')^{n-2} \int_{\overline{\mathcal{M}}_{g,1}} (\psi'')^{| \underline{e} | -n+1}
\\
=&
\langle \tau_{e_1;a_1} \dots \tau_{e_{n-1}; a_{n-1}} \tau_{e_n; a_n} \rangle
+ (-1)^n  \langle \tau_{ e_1+\dots + e_n -n+1 } \rangle .
\end{split}
\]
The first equality used the comparison of the $\psi$ classes above.
For the second equality the following ingredients are used: the projection formula, the self-intersection formula of $\Delta$ above, 
the facts that $\dim \Mbar_{0,n+1} = n-2$ and
\[
\pi^* \psi_i |_{\Delta} = \psi''|_{\Delta} .
\]
For the third equality we refer to Lemma~\ref{l:1} and the fact
\[
 \int_{\Mbar_{0,n+1}} \psi_{n+1}^{n-2} =1.
\]
This concludes the proof.
\end{proof}

Let us introduce the following notations.
For any $J \subset  \{ 1,\ldots, n \}$, let 
$\underline{a}_J$ be the ordered sublist of $\underline{a}$, consisting of sub-indices in $J$ and
\[
  |{\underline{a}_J}| : = \sum_{j\in J}  a_j., \qquad \# (\underline{a}_J ) :=  \# (J) = \op{cardinality}(J).
\]

\begin{lemma}[Chambers and walls of weights] \label{l:3}
$\langle \tau_{\underline{e}; \, \underline{a}} \rangle = \langle \tau_{\underline{e}; \, \underline{a'}} \rangle$ if $\mathcal{P}(\underline{a}) = \mathcal{P}(\underline{a}')$.
In particular, if $| \underline{a}_J | \neq 1$ or $1^+$ for any $J \subset  \{ 1,\dots,n-1 \}$, then a small perturbation of weights 
$\underline{a}$ does not change $\langle \tau_{\underline{e}; \, \underline{a}} \rangle$.
\end{lemma}

\begin{proof}
As in Lemma~\ref{l:1}, if $a_j \geq a'_j$ for all $j$, there is a birational morphism
\[
\pi: \overline{ \mathcal{M} }_{g, a_1\dots a_n}  \rightarrow \overline{ \mathcal{M} }_{g, a'_1\dots a'_n}.
\]
Furthermore, $\psi_j = \pi^* \psi_j$ if and only if $\mathcal{P}(\underline{a}) = \mathcal{P}(\underline{a}')$.
It is not difficult to see that the space of weight vectors $I^n \subset [0^+, 1]^n$ is divided into chambers by linear faces defined by 
$| \underline{a}_J | = 1$ or $1^+$ for some $J$ and $\langle \tau_{\underline{e}; \, \underline{a}} \rangle = \langle \tau_{\underline{e}; \, \underline{a'}} \rangle$ if $\underline{a}$ and $\underline{a}'$ belong to the same chamber.
The lemma now follows.

Alternatively, there are birational morphisms
\[
\pi: \Mbar_{g,n} \to \overline{ \mathcal{M} }_{g, a_1\dots a_n}, \quad \pi': \Mbar_{g,n}  \rightarrow \overline{ \mathcal{M} }_{g, a'_1\dots a'_n}
\]
such that
\[
 \psi_i = \pi^* (\psi_i ) + \Delta_i, \quad  \psi_i = \pi^* (\psi'_i ) + \Delta'_i .
\]
Because of $\mathcal{P}(\underline{a}) = \mathcal{P}(\underline{a}')$, $\Delta_i = \Delta'_i$ for all $i$.
Hence $\langle \tau_{\underline{e}; \, \underline{a}} \rangle = \langle \tau_{\underline{e}; \, \underline{a'}} \rangle$.
\end{proof}

\begin{lemma} \label{l:4}
\[
\begin{split}
\langle \tau_{e_1; a_1} \dots \tau_{e_n; a_n} \rangle 
= 
 \sum_{ \substack{ J \subset \{1,\dots, n-1\} \\ |\underline{a}_J| + a_n \leq 1 } } (-1)^{\# (J)}
\langle   \tau_{\underline{e}_{J^c}; \, \underline{a}_{J^c}} \tau_{e_n + |\underline{e}_J| - \# (J) } \rangle.
\end{split}
\]
\end{lemma}

\begin{proof}
It is not difficult to see that we can choose $a_1',\dots,a_n'$ such that 
\begin{enumerate}
\item[(i)] $\mathcal{P}(\underline{a}) = \mathcal{P}(\underline{a}')$;
\item[(ii)] $|\underline{a}'_J |$ are all different for different $J \subset \{ 1,\dots,n-1 \}$.
\end{enumerate}
By (i) and Lemma~\ref{l:3}, $\langle \tau_{\underline{e};\, \underline{a}} \rangle = \langle \tau_{\underline{e}; \, \underline{a}'} \rangle$.
Now, gradually increase $a'_n$ to $1$. By (ii) above, each wall-crossing is of the type in Lemma~\ref{l:2}. 
Apply it repeatedly we have
\[
\begin{split}
 \langle \tau_{e_1;a_1} \dots \tau_{e_n; a_n} \rangle
&=
\langle \tau_{e_1;a_1'} \dots \tau_{e_n; a_n'} \rangle
\\
&=
\sum_{ \substack{ J \subset \{1,\dots, n-1\} \\ | \underline{a}'_J | + a_n' \leq 1 } } 
(-1)^{ \# (J) }
\langle  \tau_{\underline{e}_{J^c}; \, \underline{a}_{J^c}' }  \tau_{e_n + |e_J| - \#(J) } \rangle
\\
&=
\sum_{ \substack{ J \subset \{1,\dots, n-1\} \\ | \underline{a}_J | + a_n \leq 1 } } 
(-1)^{ \# (J) }
\langle  \tau_{\underline{e}_{J^c}; \, \underline{a}_{J^c} }  \tau_{e_n + |e_J| - \#(J) } \rangle ,
\end{split}
\]
where Lemma~\ref{l:3} is used in the last equality.
\end{proof}

\begin{proof}[Proof of Proposition~\ref{p:1}]
Finally, we are ready to prove the proposition by induction on the number of marked points with weight $<1$.
\[
\begin{split}
&\langle \tau_{e_1; a_1} \dots \tau_{e_n; a_n} \rangle   \\
=&
\sum_{ \substack{ J \subset \{1,\dots, n-1\} \\ | \underline{a}_J | + a_n \leq 1 } } 
(-1)^{\# (J)}
\langle  \tau_{\underline{e}_{J^c}; \, \underline{a}_{J^c} }  \tau_{e_n + |e_J| - \#(J) } \rangle 
\\
= &
\sum_{ \substack{ J \subset \{1,\dots, n-1\} \\ | \underline{a}_J | + a_n \leq 1 } } 
\sum_{p \in \mathcal{P}(  \underline{a}_{J^c},  \underline{a}_J +a_n  ) } 
(-1)^{ \op{codim}(p) }
\langle  \tau_{p \left( \underline{e}_{J^c},\,  |\underline{e}_J | - \# (J) + e_n  ;\, \underline{a}_{J^c},\, |\underline{a}_J |+ a_n  \right)}  \rangle 
\\
=&
\sum_{p\in \mathcal{P}(a_1,\dots,a_n)} (-1)^{\op{codim}(p)}
\langle \tau_{ p(\underline{e} ; \, \underline{a})  } \rangle .
\end{split}
\]
Here in the second equality, the induction hypothesis is used.
The thrid equality holds as 
\[
 \mathcal{P}(a_1,\dots,a_n) = \bigsqcup_{ \substack{ J \subset \{1,\dots, n-1\} \\ | \underline{a}_J | + a_n \leq 1 } }  \mathcal{P}(  \underline{a}_{J^c} , \, \underline{a}_J +a_n ) .
\]
Note that here on the RHS, $\mathcal{P}(  \underline{a}_{J^c} , \, \underline{a}_J +a_n )$ is identified as a subset of $\mathcal{P}(a_1,\dots,a_{n})$ by appending $(\underline{a}_J a_n)$ at the end of each partition of $\underline{a}_{J^c}$.
We also note that for brevity we abuse the notation above and used $\mathcal{P}(\underline{a})$ for $\mathcal{P}(\underline{e}; \underline{a})$.
%$p \in \mathcal{P}(a_1,\dots,a_n)$ should be either of type $p_1$ or of type $p_2$. It only depends on whether $a_n$ pairs with others or not.
\end{proof}

%Here we give a corollary of wall-crossing formula:
\begin{corollary} \label{c:1}
For $k \geq -1$, $ \underline{e} \in \mathbb{N}^n $, and $( b, \underline{a} )\in \A^{n+1}$, we have
\[
\langle
\tau_{k+1;b} \tau_{\underline{e}; \, \underline{a}}
\rangle
= 
\langle
\tau_{k+1} \tau_{\underline{e}; \, \underline{a}}
\rangle
-  \sum_{ \substack{ I \subset \{1,\dots, n\}   \\   b + | \underline{a}_I |  \leq 1 }  }
\langle
\tau_{ k+1+| \underline{e}_I | - |I|  ;\, b + | \underline{a}_I | } \tau_{ \underline{e}_{I^c} ; \, \underline{a}_{I^c} }
\rangle.
\]
\end{corollary}
\begin{proof}
As in the proof of Proposition~\ref{p:1}, notice that 
\[
\begin{split}
\mathcal{P}(b, \underline{a}) 
&=\bigsqcup_{ \substack{ J \subset \{1,\dots, n\} \\ b +   |\underline{a}_J|   \leq 1 } }  \mathcal{P}( b+ \underline{a}_J   , \, \underline{a}_{J^c} )
\\
&= \mathcal{P}(\underline{a}) \sqcup
\left(
\sqcup _{ \substack{  \phi \neq J \subset \{ 1,\dots,n \}  \\  b+| \underline{a}_J | \leq 1  } }
\mathcal{P}( b+ \underline{a}_J , \, \underline{a}_{J^c}  )
\right).
\end{split}
\]
Here $\mathcal{P}(\underline{a})$ is identified as a subset of $\mathcal{P}(b,\underline{a})$ by appending $(b)$ at the end of each partition of $\underline{a}$.

Let 
\[
\mathcal{P}_1 := \mathcal{P}(\underline{a}),
\qquad
\mathcal{P}_2 := \sqcup _{ \substack{  \phi \neq J \subset \{ 1,\dots,n \}  \\  b+| \underline{a}_J | \leq 1  } }
\mathcal{P}( b+ \underline{a}_J , \, \underline{a}_{J^c}  )
\]
for abbreviation and apply Proposition~\ref{p:1} to the left hand side of Corollary~\ref{c:1}
\[
\begin{split}
& \langle \tau_{k+1;b} \tau_{\underline{e}; \, \underline{a}}
\rangle \\
=& \sum_{p_1 \in \mathcal{P}_1} 
(-1)^{{\rm codim}(p_1)}  
\langle
\tau_{ p_1(k+1,\underline{e};\,  b,\underline{a} )   }
\rangle
+ \sum_{p_2 \in \mathcal{P}_2} 
(-1)^{{\rm codim}(p_2)}  
\langle
\tau_{ p_2(k+1,\underline{e}; \,  b,\underline{a} )   }
\rangle
\\
= 
&\langle
\tau_{k+1} \tau_{\underline{e}; \, \underline{a}}
\rangle
+ \sum_{p_2 \in \mathcal{P}_2} 
(-1)^{{\rm codim}(p_2)}  
\langle
\tau_{ p_2(k+1,\underline{e}; \, b,\underline{a} )   }
\rangle.
\end{split}
\]
It remains to show that 
\[
\sum_{p_2\in \mathcal{P}_2}
(-1)^{{\rm codim}(p_2)}
\langle
\tau_{ p_2(k+1,\underline{e}; \, b,\underline{a} )   }
\rangle
=
-  \sum_{ \substack{ I \subset \{1,\dots, n\}   \\   b + | \underline{a}_I | \leq 1  }  }
\langle
\tau_{ k+1+ |\underline{e}_I| - |I| ; \, b + | \underline{a}_I | } \tau_{ \underline{e}_{I^c} ;\, \underline{a}_{I^c} }
\rangle.
\]
This can be checked by applying Proposition~\ref{p:1} to the RHS.
The proof is now complete.
%Expanding the RHS by Proposition~\ref{p:1} and comparing the coefficient of 
%$\langle  \tau_{ p_2(k+1,\underline{e}; b,\underline{a} )   }  \rangle$ for fixed $p_2 \in \mathcal{P}_2$, we have
%\begin{enumerate}
%\item 
%%\[
%$\langle  \tau_{ p_2(k+1,\underline{e}; b,\underline{a} )   }  \rangle [{\rm LHS}] 
%=  (-1)^{{\rm codim}(p_2)} .$
%%\]
%\item For the RHS, we give a notation first. We write $p_2(b,\underline{a}) = (b a_1 \dots a_m) (...)(...)$
%\[
%\begin{split}
%\langle  \tau_{ p_2(k+1,\underline{e}; b,\underline{a} )   }  \rangle [{\rm RHS}]
%&= \langle  \tau_{ p_2(k+1,\underline{e}; b,\underline{a} )   }  \rangle
%\left[
%-\sum_{\phi \neq J \subset \{1,\dots, m\}}  
%\langle
%\tau_{ k+1+ |\underline{e}_J| - |J| ; b + | \underline{a}_J | } \tau_{ \underline{e}_{J^c} ; \underline{a}_{J^c} }
%\rangle
%\right]
%\\
%&= -\sum_{\phi \neq J \subset \{1,\dots, m\}} (-1)^{{\rm codim}(p_2) - |J| }
%= (-1)^{ {\rm codim}(p_2) }.
%\end{split}
%\]
%\end{enumerate}
%This computation holds for all $p_2 \in \mathcal{P}_2$ and we complete the proof.
\end{proof}

\subsection{Correlators with unstable components} \label{s:1.2}

For future reference, we introduce the following \emph{convention}
\begin{equation} \label{e:unstable}
\langle \tau_{k_1;a_1} \dots \tau_{k_n;a_n}  \rangle_{0,n}
:= \sum_{p \in \mathcal{P}(\underline{a})}  (-1)^{ \op{codim} (p) }
\langle \tau_{p( \underline{k}; \, \underline{a} )} \rangle
\end{equation}
for the \emph{unstable} case $2g -2 + a_1 + \dots + a_n \leq 0$ via Proposition~\ref{p:1}.
We note that even though the unstable correlators on the RHS (weight all equal to $1$) vanish by definition,
unstable correlators for other weights might not vanish.
Also, the weighted unstable correlators do not correspond to integration over Hassett's moduli, which by definition is empty.
For example, 
\[
 \langle \tau_{0; 1/3} \tau_{0; 1/3} \tau_{0; 1/3} \rangle_{0,3} =  \langle \tau_{0} \tau_{0} \tau_{0} \rangle_{0,3} =1.
\]

We now introduce the generating functions of these correlators.
Let the ``phase space'' $\mathbf{H}^{\A}$ be  $\mathbb{C}^{\mathbb{N} \times \A}$, where  $\A \subset [0^+, 1]$ is an additively closed subset in the following sense: $\forall a_1, a_2 \in \A$, if $a_1 +a_2 \in [0^+, 1]$ then $a_1+a_2 \in \A$.
Basic examples of such additively closed $\A$ of finite cardinality include $\A= \{ 0^+ \},  \{ 0^+, 1 \}$, $\{ 1 \}$, $\{ \frac{1}{2}, 1 \}$ and $[0^+, 1]$.

Let 
\[
 t := \sum_{k \in \mathbb{N}, a \in \A} t_{k;a} z^k \mathbf{e}^a \in \mathbf{H}^{\A},
 \]
where $\{ z^k \mathbf{e}^a \}$ denote the ``standard basis'' of $\mathbf{H}$.
Introduce the generating function
\begin{equation} \label{e:2}
F^{\A} (t) := \sum_{ n } \frac{1}{n!} \langle t^{\otimes n} \rangle
= \sum_{n} \frac{1}{n!} \sum_{ ( \underline{k}, \underline{a} ) \in \mathbb{N}^n \times \A^n }
t_{\underline{k}; \underline{a}} \langle  \tau_{\underline{k}; \, \underline{a}} \rangle =: \sum_{g=0}^{\infty} F_g^{\A},
\end{equation}
where $F_g^{\A}$ is the sum of all correlators with genus $g$, as defined by \eqref{e:g}.
A choice of $\A$ is made throughout.
%\emph{Unless necessary, the superscript $\A$ in $F^{\A}$ will be omitted.}

For example, if $a_i=1$ for all $i$, we have the ``usual'' Witten--Kontsevich generating function \cite{eW}
\begin{equation} \label{e:3}
F (t) = \left( \sum_{n} \frac{1}{n!} \sum_{ \underline{k} \in \mathbb{N}^n }
t_{k_1} \dots t_{k_n} \langle \tau_{ \underline{k}} \rangle \right).
\end{equation}
Here we use the convention 
\[
t_k := t_{k;1} , \quad \tau_{k} := \tau_{k;1}
\]
for the weight $1$ markings.

\section{Virasoro constraints} \label{s:2}

\subsection{A combinatorial function $h$}
%In this section, we recall virasoro constraints for unweighted case and introduce $h$ function for the later convenience. 

Given $k, e \in \mathbb{N}$, we define
\[
h_{k; e} := \frac{(2k+2e+1)!!}{(2e-1)!!}.
\]
For general $ \underline{e} \in \mathbb{N}^n$, define 
\[
h_{k;  \underline{e}} := 
\sum_{ \phi \neq J \subset \{1,\dots,n \} } 
(-1)^{ \# (J) - 1 } h_{  k;  \ | \underline{e}_J | - \#(J) + 1}  . %,
\]
%where  $\underline{e}_J$ is the ordered sublist of $\underline{e}$, consisting of sub-indices in $J$ and
%\[
%  |{\underline{e}_J}| : = \sum_{j\in J}  e_j., \qquad  \# (J) = \op{cardinality}(J).
%\]
For example:
\[
\begin{split}
h_{k; i,j} := &h_{k; i} + h_{k;j} - h_{k;i+j-1}  \\
= & \frac{(2k+2i+1)!!}{(2i-1)!!} + \frac{(2k+2j+1)!!}{(2j-1)!!} - \frac{(2k+2i+2j-1)!!}{(2i+2j-3)!!}.
\end{split}
\]

For the future reference, we write $\underline{e}' \leq \underline{e}$ if $\underline{e}'$ is a sublist of $\underline{e}$;
we write $\underline{e}' + \underline{e}'' = \underline{e}$ if $\underline{e}''$ is the complementary sublist of $\underline{e}' \leq \underline{e}$; the notation $(\underline{e}')^c$ is also used for the complementary sublist.
We say $\underline{e}' \sim \underline{e}''$ if the underlying unordered sublists , i.e., subsets, of $\underline{e}'$ and $\underline{e}''$ are identical.
Denote by $\op{Power}(\underline{e})$ the equivalence classes of the sublists of $\underline{e}$. 

We can similarly define Power$(\underline{e}, \underline{a})$ to be the equivalence classes of the sublists of the list of 2-tuples $(\underline{e}, \underline{a}) = ( [e_1a_1], [e_2a_2], \dots, [e_na_n] )$. Here $\underline{a}$ just goes along for the ride. Same definition works for the list of $n$-tuples.

\begin{lemma} \label{l:h}
For $\underline{e} \in \mathbb{N}^n$ and $\underline{a} \in \A^n$, we have
\begin{equation} \label{e:h}
 h_{k;\underline{e}} = \sum_{[ \underline{e}', \underline{a}' ] \in \op{Power}( \underline{e}, \underline{a})} (-1)^{\# (\underline{e}') -1}  \frac{ \# \left( \op{ Aut} ( \underline{e}, \underline{a}) \right)} { \# ( \op{ Aut} ( \underline{e}', \underline{a}') ) \#  ( \op{ Aut} ( \underline{e}', \underline{a}' )^c) } 
 h_{k; | \underline{e}' | - \# (\underline{e}') +1 },
\end{equation}
where the summation is taken over all equivalence classes once.

In particular, if $a_1=a_2= ... = a_n$, the above equation reduces to the following form:
\begin{equation}
 h_{k;\underline{e}} = \sum_{[ \underline{e}' ] \in \op{Power}( \underline{e})} (-1)^{\# (\underline{e}') -1}  \frac{ \# \left( \op{ Aut} ( \underline{e}) \right)} { \# ( \op{ Aut} ( \underline{e}') ) \#  ( \op{ Aut} ( \underline{e}')^c) } 
 h_{k; | \underline{e}' | - \# (\underline{e}') +1 },
\end{equation}
\end{lemma}

\begin{proof}
We note that by definition $ h_{k; | \underline{e}' | - \# (\underline{e}') +1 } =  h_{k; | \underline{e}' | - \# (\underline{e}'') +1 }$ if $(\underline{e}', \underline{a}') \sim (\underline{e}'', \underline{a}'')$.
It is easy to see that the coefficients $\frac{ \# \left( \op{ Aut} ( \underline{e}, \underline{a}) \right)} { \# ( \op{ Aut} ( \underline{e}', \underline{a}') ) \#  ( \op{ Aut} ( \underline{e}', \underline{a}' )^c) } $ is the number of  ordered sublists in the same equivalence class as $(\underline{e}', \underline{a}')$.
\end{proof}

%Since$h_{k; \underline{e}}$ does not depend on the order of the list $\underline{e}$, we sometimes also use the notation
%$h_{k;J}$ for $J$ a set of numbers.

\subsection{Virasoro constraints for unweighted case}
To motivate the general weighted Virasoro, it might be beneficial to recall the ``unweighted'' case, i.e., all marked points are of weight $1$.
\begin{lemma}[Virasoro operator] \label{l:5}
We define the following differential operators:
\[
\begin{split}
 2 L_{-1}  &=
- \frac{\partial}{\partial t_0} 
+ \sum_{i= 1}^{\infty}  t_i \frac{ \partial }{ \partial t_{i-1} } + \frac{t_0^2}{2},
\\
 2 L_0 & = 
-3 \frac{\partial}{\partial t_1} 
+ \sum_{i=0}^{\infty} (2i+1) t_i \frac{ \partial }{ \partial t_i } + \frac{1}{8} ,
\\
 2 L_k  &= 
-(2k+3)!! \frac{ \partial }{ \partial t_{k+1} } 
+ \sum_{i=0}^{\infty} h_{k;i} \ t_i \frac{ \partial }{ \partial t_{i+k} } 
\\
 &+ \frac{1}{2} \sum_{ \substack{ r+s=k-1 \\ r,s \geq 0 } } (2r+1)!! (2s+1)!! 
\quad \frac{ \partial^2 }{ \partial t_r  \partial t_s }, \quad {\rm for} \ k \geq 1.
\end{split}
\]
These operators satisfy the following \emph{Virasoro relations}
\begin{equation} \label{e:vir}
[ L_m, L_n ] = (m-n) L_{m+n}, \quad {\rm for } \ m, n \geq -1.
\end{equation}
Furthermore, let $F(t)$ be the weight $1$ generating function, then the following \emph{Virasoro constraints} hold
\[
 L_k ( e^{F} )=0 \quad \mbox{ for } k \geq -1.
\]
\end{lemma}

%\begin{proof}
The Virasoro relations can be checked with straightforward computation from the definition of the operators $L_k$. 
The Virasoro constraints on $e^F$ are equivalent to Kontsevich's theorem.
%\end{proof}
It is well known and easy to check that the Virasoro constraints above are equivalent to the following recursion relations. 

\begin{proposition} \label{p:2}  %[Witten's Conjecture with negative $\psi$ power]
For $k \geq -1$, $\underline{e} \in \mathbb{N}^n$, we have the following recursive relation:
\[
\begin{split}
 &\langle  \tau_{k+1}  \tau_{ \underline{e} } \rangle_{g, n+1} = \frac{1}{ (2k+3)!! } 
  \left\{ \sum_{j=1}^n h_{ k; e_j } \langle \tau_{\underline{e}+ k[j]} \rangle \right.
\\
 +&\frac{1}{2} \sum_{ \substack{ r+s=k-1 \\ r,s \geq 0 } } (2r+1)!!(2s+1)!! 
\left. \left[ \langle \tau_r \tau_s \tau_{\underline{e}} \rangle_{g-1,n+2} 
+ \sum_{  I \subset \{1,\dots n\}  }  \langle \tau_r \tau_{\underline{e}_I} \rangle \langle \tau_s \tau_{\underline{e}_{I^c}} \rangle \right] \right\} ,
\end{split}
\]
where 
\[
 \langle \tau_{ \underline{e}+k[j] } \rangle := \langle \tau_{e_1}\dots \tau_{e_j+k} \dots \tau_{e_n}  \rangle .
 \]
\end{proposition}

\subsection{Recursions for the weighted pointed correlators}
In order to write down the operators $L_k$, we start with the recursion relation for generating function with weighted points,
generalizing Proposition~\ref{p:2}.

%\begin{theorem}[Generalized Witten conjecture]  \label{t:3.1}
\begin{theorem}  \label{t:3.1}
For $k \geq -1$, $\underline{e} \in \mathbb{N}^n$, and $(b, \underline{a}) \in {\A}^{n+1}$, the following recursions hold for weighted pointed correlators
\[
\begin{split}
& \langle  \tau_{k+1; \, b}  \tau_{ \underline{e}; \,  \underline{a} }  \rangle_{g, n+1} =
- \sum_{ \substack{ \phi \neq  I \subset \{ 1,\dots, n \} \\ b + |\underline{a}_I | \leq 1 } }
\langle 
\tau_{k+1 + | \underline{e}_I | - | I | ; \, b + | \underline{a}_I | }
\tau_{\underline{e}_{I^c}; \, \underline{a}_{I^c}}
\rangle_{g, n - \#(I) + 1}
\\
& + \frac{1}{ (2k+3)!! } 
\left\{ \sum_{  \substack{ \phi \neq J \subset \{ 1,\dots, n \} \\  | \underline{a}_J | \leq 1}  }
h_{k; \, \underline{e}_J} 
\langle   \tau_{ | \underline{e}_J | - \#(J) + 1 + k;   \, | \underline{a}_J | } \ \tau_{ \underline{e}_{J^c}; \,  \underline{a}_{J^c} }  \rangle_{g, n-\#(J)+1} \right.
\\
+&\frac{1}{2} \sum_{ \substack{ r+s=k-1 \\ r,s \geq 0 } } (2r+1)!!(2s+1)!! 
\left. \left[ 
\langle \tau_r \tau_s \tau_{\underline{e}; \, \underline{a}} \rangle_{g-1,n+2} 
+ \sum_{  I \subset \{1,\dots, n\}  }  
\langle \tau_r \tau_{\underline{e}_I; \, \underline{a}_I} \rangle 
\langle \tau_s \tau_{\underline{e}_{I^c};\, \underline{a}_{I^c}} \rangle \right] \right\}
.
\end{split}
\]

This system of recursions uniquely determines $F(t)$ up to the initial conditions:
\[
\langle \tau_{1;a} \rangle = \frac{1}{24}, \quad \langle \tau_{0;a_1}\tau_{0;a_2}\tau_{0;a_3} \rangle = 1
\quad\mbox{ for any $a, a_1, a_2, a_3 \in \A$ }.
\]
\end{theorem}
%\end{theorem}

The proof of this theorem will be given in the next section.
We note that when all weights are equal to $1$, Theorem~\ref{t:3.1} is reduced to Proposition~\ref{p:2} as $\# (J) $ in the second summand must be equal to $1$.

\subsection{Structure of the generalized Virasoro relations} \label{s:2.4}
We will define the Virasoro operators $\{ L_{k;a}^{\A} \}_{k \geq -1}$ for a given $a\in \A$.
These Virasoro operators will form a more general Lie algebra.
Before we proceed, we first make some general remarks.

\begin{lemma} \label{l:3.2}
Assume that the differential operators $\{ L_k \}$ for $k \geq -1$ satisfy the Virasoro relations \eqref{e:vir}.
Let $\{ M_{k;a} \}$ for $k \geq -1$ and  $a \in \A$ be the differential operators satisfying the following relations
\begin{equation} \label{e:3.1}
%\begin{split}
[ M_{k_1; a_1} , M_{k_2; a_2} ] = 0, 
\qquad [L_{k_1}, M_{k_2; a}] = - (k_2+c ) M_{k_1+k_2; a}
%\qquad [L_{k_1}, L_{k_2}] = (k_1-k_2) L_{k_1+k_2}.
%\end{split}
\end{equation}
for a constant $c$ and for all $k_1, k_2, a, a_1, a_2$.
Let $\epsilon_a$ be any constants and define the deformed operators
\[
 L_{k; a} := L_k + \epsilon_a M_{k; a}.
\]
We have
\begin{equation} \label{e:3.2}
[ L_{k_1; a_1}, L_{k_2; a_2} ] = \left( k_1 + c \right) L_{k_1 + k_2; a_2} - \left( k_2 + c \right) L_{k_1 + k_2; a_1}
%(k_1-k_2) L_{k_1+k_2} + \frac{2k_1+1}{2}  M_{k_1+k_2; b} -  \frac{2k_2+1}{2}  M_{k_1+k_2; a},
\end{equation}
for all $k_1, k_2, a_1, a_2$.
In particular, for a fixed $a$, $\{ L_{k;a} \}_{k \geq -1}$ satisfy the ``usual'' Virasoro relations \eqref{e:vir}.
\end{lemma}

\begin{proof}
Equation \eqref{e:3.2} follows from \eqref{e:3.1} by a straightforward calculation.
The last sentence is a simple observation.
\end{proof}

\begin{corollary} \label{c:2.6}
Suppose that $\{ L_{k;a} \}$ for $k \geq -1$ and $a \in \A$ satisfy the commutation relations \eqref{e:3.2}.
The ``deformed operators'' $\{ L_{k;a} + \epsilon_a M_{k;a} \}$ satisfy the same commutation relations if and only if
\begin{equation} \label{e:3.3}
[ M_{k_1; a_1} , M_{k_2; a_2} ] = 0, 
\qquad [L_{k_1; a_1}, M_{k_2; a_2}] = - (k_2+c ) M_{k_1+k_2; a_2}
\end{equation}
for all $k, a$.
\end{corollary}

\begin{remark} \label{r:3.4}
We consider the vector space $V_{\A}$ generated by the operators 
%\[
$ \left\{ L_{k; a} \right\}_{k \geq -1, a\in \A} $,
%\] 
which forms a \emph{Lie algebra} with the commutation relations described in \eqref{e:3.2}.
This Lie algebra in particular contains "$\A$ copies" of the simple Lie algebras $\{ L_{k; a} \}_{k \geq -1, a\in \A}$, representing ``half'' of the Virasoro algebra. It has a proper abelian ideal $I_{\A}$ generated by $\{ L_{k; a_1} - L_{k; a_2}  \}_{k \geq -1, a_1,a_2\in \A}$, which satisfies the commutation relations in \eqref{e:3.3}. 
There is a short exact sequence
\[
 0 \to I_{\A} \to V_{\A} \to V \to 0,
\]
where $V$ is the ``usual'' centerless Virasoro algebra (with $k \geq -1$).
In other words, $V_{\A}$ is a semi-direct product of $V$ with an abelian Lie algebra $I_{\A}$.
It is well known that the highest-weight representations of $V_{\A}$ are all induced from those of $V$, cf.\ \cite{CK}. %See, e.g., \cite{CK}.
This suggests that the ``generalized'' Virasoro is still related to the KdV hierarchy (instead of its variants).
We will spell this out explicitly in Section~\ref{s:KdV}.
\end{remark}

\subsection{Virasoro operators for the weighted pointed curves}
\begin{lemma}\label{l:3.5}
Assume $1 \in \A$. Define the following differential operators
\[
\begin{split}
2L_{-1} 
&:= - \frac{ \partial }{ \partial t_0 } 
+ \sum_{ n \geq 1 } 
\frac{1}{n!}
\sum_{ \substack{ (\underline{e}, \underline{a}) \in (\mathbb{N} \times \A)^n \\ |\underline{a}| \leq 1,\ |\underline{e}|-n \geq 0 }  }
 t_{ \underline{e};  \underline{a} } 
\frac{ \partial } { \partial t_{ | \underline{e} | - n ;  | \underline{a} | } } 
\\
&\qquad\qquad + \frac{1}{2}
\sum_{m,n \geq 1} \frac{(-1)^{m+n}}{m!n!}
\sum_{ \substack{ (\underline{e}_1, \underline{a}_1) \in (\mathbb{N} \times \A )^m \\ |\underline{a}_1| \leq 1,\ |\underline{e}_1|-m+1 = 0 }  }
\sum_{ \substack{ (\underline{e}_2, \underline{a}_2) \in (\mathbb{N} \times \A )^n \\ |\underline{a}_2| \leq 1,\ |\underline{e}_2|-n+1 = 0 }  }
t_{\underline{e}_1;\underline{a}_1}
t_{\underline{e}_2; \underline{a}_2}
\\
2L_0 
& := -3 \frac{ \partial }{ \partial t_1 } 
+ \sum_{ n \geq 1 } \frac{1}{n!} 
\sum_{ \substack{ (\underline{e}, \underline{a}) \in (\mathbb{N} \times \A)^n \\ |\underline{a}| \leq 1,\ |\underline{e}|-n+1 \geq 0 }  }
h_{ 0; \underline{e} } \,
t_{ \underline{e};  \underline{a} } 
\frac{ \partial }{ \partial t_{ | \underline{e} | - n+1;  | \underline{a} | } } 
+ \frac{1}{8}, 
\\
2L_k 
& := -(2k+3)!! \frac{ \partial }{ \partial t_{k+1} } 
+ \sum_{ n \geq 1 } 
\frac{1}{n!} 
\sum_{ \substack{ (\underline{e}, \underline{a}) \in (\mathbb{N} \times \A )^n \\ |\underline{a}| \leq 1,\ |\underline{e}|-n+1+k \geq 0 }  }
h_{ k; \underline{e} } \, t_{ \underline{e};  \underline{a} } 
\frac{ \partial }{ \partial t_{ |\underline{e}| - n+1+k;  | \underline{a} | } } \\
 & \qquad +
\frac{1}{2}
\sum_{ \substack{ r+s=k-1 \\ r,s \geq 0 } }  (2r+1)!! (2s+1)!! 
\frac{ \partial^2 }{ \partial t_r \partial t_s }, \quad {\rm for} \ k \geq 1,
\end{split}
\]
as well as
\[
\begin{split}
\frac{2}{(2k+3)!!} & M_{k; b}  := - \frac{ \partial }{ \partial t_{k+1; b} }
 + \frac{\partial}{\partial t_{k+1}}  \\
- & \left[
\sum_{m \geq 1}  \frac{1}{m!}
\sum_{ \substack{ ( \underline{e} \times \underline{a} ) \in (\mathbb{N} \times \A )^m \\ |\underline{a}| + b \leq 1  , \ |\underline{e}| - m + k +1 \geq 0 }  }
 t_{ \underline{e}; \underline{a}  } 
\frac{\partial}{\partial t_{ k + |\underline{e}| - m + 1 ; |\underline{a}| +b }}
\right].
\end{split}
\]
These operators satisfy the Virasoro relations \eqref{e:vir} as well as the commutation relations \eqref{e:3.1} with $c= \frac{3}{2}$.
\end{lemma}

The proof of this lemma is also deferred to the next section. 
We are now ready to spell out the Virasoro operator $L_{k;a}$ for $a \in \A$.
\begin{equation} \label{e:lka}
L_{k;a} := L_k + M_{k;a}.
\end{equation}

\begin{corollary} \label{c:2.9}
Assume $1 \in \A$. 
$L_{k;a}$ satisfy the commutation relations \eqref{e:3.2} with $c= \frac{3}{2}$.
\end{corollary}

\begin{proof}
The statement is an immediate consequence of Lemma~\ref{l:3.2}. The commutation relations needed is stated in Lemma~\ref{l:3.5}.
\end{proof}

When $1 \notin \A$, we wish to define the Virasoro operators $L^{\A}_{k;a}$ similar to $L_{k;a}$ in \eqref{e:lka}, 
but with the condition that $L^{\A}_{k;a}$ should involve only $t_{l;b}$ for $b \in \A$, i.e., only the variables in the phase space $\mathbf{H}^{\A}$.
We note that $L_{k}$ and $M_{k;a}$ both involve differentiation with respect to $t_k$ when the weight $1$ might not be in $\A$.

%If $1 \in \A$, then nothing needs to be done. Assume that $1 \notin \A$. 
Now assume $1 \notin \A$ and let $\bar{\A} := \A \cup \{ 1 \}$.
Rephrasing Corollary~\ref{c:1}, we have
\begin{equation} \label{e:2.7}
 M_{k;a} F^{\bar{\A}} (t) =0.
\end{equation}
We will now introduce a new coordinate system $\{ \mathbf{t}_{k;a} \}$ on the phase space.
Let
\begin{equation} \label{e:2.8}
  \begin{split}
%\partial_{\mathbf{t}_{k+1;b}} 
 \mathbf{v}_{k+1;b} := \frac{\partial}{\partial \mathbf{t}_{k+1;b}} := \frac{ \partial }{ \partial t_{k+1; b} }
+ \left[
\sum_{m \geq 1}  \frac{1}{m!}
\sum_{ \substack{ ( \underline{e} \times \underline{a} ) \in (\mathbb{N} \times \A )^m \\ |\underline{a}| + b \leq 1  , \ |\underline{e}| - m + k +1 \geq 0 }  }
 t_{ \underline{e}; \underline{a}  } 
\frac{\partial}{\partial t_{ k + |\underline{e}| - m + 1 ; |\underline{a}| +b }}
\right].
\end{split}
\end{equation}

\begin{lemma} \label{l:4.1}
The formal vector fields $\{  \mathbf{v}_{k;a} \}$ commute pairwise
\[
 \left[  \mathbf{v}_{k;a} , \mathbf{v}_{k';a} \right] =0, \quad \forall k, k' \geq -1,
\]
for any given weight $a$. 
\end{lemma}

\begin{proof}
This follows immediately from \eqref{e:3.3},  $[ M_{k_1; a} , M_{k_2; a} ] = 0$ and the observations that
\[
  M_{k;a} = \frac{\partial}{\partial t_{k+1}} - \mathbf{v}_{k+1;a}, \quad [\partial_{t_{k_1}} , \mathbf{v}_{k_2;a} ]=0.
\]
\end{proof}

\begin{remark} \label{r:2.12}
Lemma~\ref{l:4.1} can be interpreted as follows. 
The pairwise-commuting vector fields $\{ \mathbf{v}_{k;a}  \}$ formally integrate to 
a formal coordinate system $\{ \mathbf{t}_{k;a} \}$ on the phase space $\mathbf{H}$  (defined in Section~\ref{s:1.2}).
In other words, we are (formally) changing the coordinate systems from $\{ {t}_{k;a} \}$ to $\{ \mathbf{t}_{k;a} \}$
and $\{ \mathbf{v}_{k;a}  \}$ are the new coordinate vector fields.
Note also that 
$\mathbf{t}_{k;1} = t_{k;1}$ 
by definition.
\end{remark}

\begin{definition}
Let $a \in \A$.
The Virasoro operators $L_{k;a}^{\A}$ are defined as follows:
\begin{enumerate}
\item If $1 \in \A$, then 
$L_{k;a}^{\A} := L_{k;a} = L_k + M_{k;a}$.
\item If $1 \notin \A$, $L_{k;a}^{\A}$ is defined by replacing all $\frac{\partial}{\partial t_m}$ by $\frac{\partial}{\partial \mathbf{t}_{m;a}}$ in the operator $L_{k;a}$ and then set $t_{l;1} =0$, for $l \in \mathbb{N}$.
\end{enumerate}
\end{definition}

We note that in the above definition, $L_{k;a}^{\A}$ depends indeed only on $t_{l;b}$ for $b \in \A$,
due to the fact that $\mathbf{t}_{m;a}$ depend only on $t_{l;b}$ for $b \in \A$.
That last claim can be verified by the definition of $\mathbf{t}_{m;b}$ in \eqref{e:2.8} and the additively closed assumption on $\mathcal{A}$.

\begin{theorem} \label{t:2.13}
The generating function $F^{\A} (t)$ of the weighted pointed correlators satisfies the (generalized) Virasoro constraint
\begin{equation} \label{e:2.9}
L^{\A}_{k; a} (e^{F^{\A}}) =0,
\end{equation}
for any $k \geq -1$ and $a\in \A \subset [0^+,1]$.
The Virasoro constraints uniquely determine the generating function up to the initial conditions in Theorem~\ref{t:3.1}.
\end{theorem}

\begin{proof}
When $1 \in \A$ and hence $\bar{\A} =\A$, the equation $L_{k;a}^{\mathcal{A}} (\exp F^{\mathcal{A}})=0$ is the operator reformalism of Theorem~\ref{t:3.1}.

Otherwise, assume that $1 \notin \A$. We define $\bar{\A} = \A \cup \{ 1 \}$. One stiill has $L^{\bar{\A}}_{k; a} (\exp F^{\bar{\A}} ) =0$.
Now we apply the trick for a given $a \in \A$: Since differentiation of  with respect to $t_k$ is the same as differentiation with respect to $\mathbf{t}_{k;a}$ (in the sense of Remark~\ref{r:2.12}), we replace all $\frac{\partial}{\partial t_l}$ by $\frac{\partial}{\partial \mathbf{t}_{l;a}}$ in the operators $L_{k;a}^{\bar{\A}}$. This gives the equation 
$$\Big(L_{k;a}^{\A} + \sum_{m} t_m \frac{\partial}{\partial \mathbf{t}_{m+k;a}}\Big)(\exp F^{\bar{\A}})=0.$$ 
Note that the operator $L_{k;a}^{\A} + \sum_{m} t_m \frac{\partial}{\partial \mathbf{t}_{m+k;a}}$ does not involve the derivative with respect to weight $1$ variable. We can restrict the above equation in the subspace $\{t_{l;1}=0\}_{l\in\mathbb{Z}}$ of the phase space. The equation reduces to
\[
L_{k;a}^{\A}(\exp F^{\A})=0,
\]
which is the desired \eqref{e:2.9}.
This completes the proof.
\end{proof}

The following corollary follows from the definition of $L_{k;a}^{\A}$ and Corollary~\ref{c:2.6}.

\begin{corollary}
$L_{k;a}^{\A}$ defined above satisfy the same commutation relation \eqref{e:3.2}.
\end{corollary}

\section{Proofs of Theorem~\ref{t:3.1} and Lemma~\ref{l:3.5}} \label{s:3}

This section, whose contents consist of detailed checking of the assertions in the previous section, can be safely skipped for impatient readers.
We use the following notation in this section: $\op{Coeff} (M,P)$ stands for the coefficient of the monomial $M$ in the (differential) polynomial $P$.

 \subsection{Proof of Theorem~\ref{t:3.1}}
 
 \begin{lemma}
 The special case $b=1$ in Theorem~\ref{t:3.1} implies the case of general $b$
 \end{lemma}
 
 \begin{proof}
 Apply  Lemma~\ref{l:4} to the LHS, one gets 
\[
\begin{split}
 &\langle  \tau_{k+1;  b}  \tau_{ \underline{e}; \, \underline{a} }  \rangle_{g, n+1} \\
= &\langle \tau_{k+1}  \tau_{ \underline{e}; \, \underline{a} }  \rangle_{g, n+1}
-  \sum_{ \substack{  \phi \neq I \subset \{1,\dots, n\}   \\   b + | \underline{a}_I |  \leq 1 }  }
\langle
\tau_{ k+1+| \underline{e}_I | - |I|  ; \, b + | \underline{a}_I | } \tau_{ \underline{e}_{I^c} ;\, \underline{a}_{I^c} }
\rangle_{g, n+1- \# (I) } .
 \end{split}
\]
Apply the special case of $b=1$ to the first term, we get exactly the statement of Theorem~\ref{t:3.1}.
\end{proof}
 
 Now we are left to check the case $b=1$.
 For our convenience, we rewrite this case as the following proposition.

\begin{proposition}
For $k \geq -1$, $\underline{e} \in \mathbb{N}^n$, and $\underline{a} \in \A^n$, we have
\begin{equation} \label{eq:3.1}
\begin{split}
& \langle   \tau_{k+1} \tau_{ \underline{e};\, \underline{a} }  \rangle_{ g, n+1 } \\
= 
&
\frac{1}{ (2k+3)!! } 
\Bigg\{
\sum_{\substack{ \phi \neq J \subset \{ 1,\dots, n \} \\ | \underline{a}_J | \leq 1}}
h_{k; \, \underline{e}_J} 
\langle   \tau_{ | \underline{e}_J | - \#(J) + 1 + k; \,   | \underline{a}_J | }  \tau_{ \underline{e}_{J^c}; \,  \underline{a}_{J^c} }  \rangle_{g, n-\#(J)+1}
\\
& \qquad \qquad
+\frac{1}{2} \sum_{ \substack{ r+s=k-1 \\ r,s \geq 0 } } (2r+1)!! (2s+1)!! 
\Bigg[
\langle \tau_r \tau_s \tau_{ \underline{e}; \,  \underline{a} }\rangle_{g-1,n+2}  \\
 & \qquad \qquad \qquad   +\sum_{ I \subset \{ 1, \dots , n \} }
\langle \tau_r \tau_{\underline{e}_I ; \,  \underline{a}_I}\rangle\langle \tau_s \tau_{\underline{e}_{I^c} ; \,  \underline{a}_{I^c}}\rangle
\Bigg]  \Bigg\}.
\end{split}
\end{equation}
%In particular, if we assume $a_i =1$ for all $i$, then the above equation reduce to original Witten's Conjecture, since $J$ can only have one element in this case. 
\end{proposition}

\begin{proof}
%We recall the following formula:
%\begin{equation}\label{e:1}
%\llangle \tau_{\underline{e}, \,  \underline{a}} \rrangle 
%= \sum_{  p\in \mathcal{P} ( \underline{a} )  } 
%(-1)^{ {\rm codim} (p) }
%\llangle \tau_{ p ( \underline{e} ) } \rrangle.
%\end{equation}
%Here $\mathcal{P}$ is the admissible partition.
%
%Now we compute
\[
\begin{split}
 & \langle \tau_{ k+1 } \tau_{ \underline{e}; \,  \underline{a} } \rangle_{g,n+1} 
= \sum_{ p \in \mathcal{P} ( \underline{e}; \, \underline{a} ) } 
(-1)^{ {\rm codim} (p) } 
\langle \tau_{k+1} \tau_{ p ( \underline{e} ) } \rangle
\\
&= \frac{1}{(2k+3)!!}  \sum_{ p \in \mathcal{P} ( \underline{e}; \, \underline{a} ) } (-1)^{ {\rm codim} (p) } 
\left\{
\sum_{j=1}^{ {\rm dim} (p) } 
h_{k; p( \underline{e} )_j } 
\langle \tau_{ p (\underline{e}) + k[j] } \rangle
\right.
\\
& 
\qquad + \frac{1}{2} 
\sum_{ \substack{ r+s=k-1 \\ r,s \geq 0 } } {\small (2r+1)!!(2s+1)!! }  \ \cdot  \\
& \qquad \qquad \cdot
\sum_{ p \in \mathcal{P} ( \underline{a} ) } 
\left.
\left[
\langle \tau_r \tau_s \tau_{ p ( \underline{e} ) } \rangle 
+\sum_{  I \subset \{ 1,\dots, { \rm dim } p \}  }
\langle \tau_r \tau_{ p ( \underline{e} )_I } \rangle  \langle \tau_s \tau_{ p ( \underline{e} )_{I^c}} \rangle
\right]
\right\}
\end{split}
\]
where the first equality is a consequence of \eqref{e:1} and the second equality follows from Proposition~\ref{p:2}.
The above output consists of three groups of sums.
In order for it to be equal to the RHS of \eqref{eq:3.1}, the following three sets of equalities must hold.
\begin{equation} \label{eq:4.2}
\begin{split}
 & \sum_{ p \in \mathcal{P} ( \underline{e}; \, \underline{a} ) } (-1)^{ {\rm codim} (p) } \sum_{j=1}^{ {\rm dim} p }  h_{ k; p ( \underline{e} )_j } 
\langle  \tau_{ p ( \underline{e} ) + k[j] } \rangle \\
= &\sum_{ \substack{\phi \neq J \subset \{ 1, \dots, n \} \\ | a_J | \leq 1}}
h_{k; \,  \underline{e}_J } 
\langle  \tau_{ |\underline{e}_J| - \# (J) + 1+ k; \,  | a_J | }  \tau_{ \underline{ e}_{J^c} ; \,  \underline{ a}_{J^c} }  \rangle.
\end{split}
\end{equation}
\begin{equation} \label{e:4.3}
\sum_{ p \in \mathcal{P} (  \underline{e}; \, \underline{a} ) } 
(-1)^{ {\rm codim} (p) }
\langle  \tau_r \tau_s \tau_{ p (\underline{e}) }  \rangle
= \langle  \tau_r \tau_s \tau_{ \underline{e}; \,  \underline{a} }  \rangle.
\end{equation}
\begin{equation} \label{e:4.4}
\begin{split}
&\sum_{ p \in \mathcal{P} (  \underline{e}; \, \underline{a} ) } 
(-1)^{ {\rm codim} (p) }
\sum_{ I \subset \{ 1, \dots, {\rm dim} (p) \} }
\langle  \tau_r \tau_{ p ( \underline{e} )_I }  \rangle
\langle  \tau_s \tau_{ p (\underline{e})_{I^c} }  \rangle  \\
=&\sum_{ I \subset \{ 1, \dots n \} }
\langle  \tau_r \tau_{ \underline{e}_I ; \,  \underline{a}_I }  \rangle
\langle  \tau_s \tau_{ \underline{ e }_{I^c} ; \,  \underline{ a }_{I^c} }  \rangle.
\end{split}
\end{equation}

To check the validity of equation \eqref{eq:4.2} we expand right hand side by equation (\ref{e:1}) and compare the coefficient of the term $\langle \tau_{ p(\underline{e}) + k[j] } \rangle$. We have
\begin{enumerate}
\item[LHS] = $ (-1)^{ {\rm codim} (p) } h_{ k; p ( \underline{e} )_j }  $,
\item[RHS] = $\displaystyle \sum_{ \phi \neq J \subset \{ j_1, \dots, j_l \} } (-1)^{ {\rm codim}(p) + \#(J) -1 } h_{k; \,  \underline{e}_J } $.
\end{enumerate}
Write $p(\underline{e})_j =: e_{j_1} + \cdots + e_{j_l} - l +1$. % and denote by $I := \{ i_1,\dots, i_n \}$.
It remains to show that RHS= LHS. We expand the RHS by the definition of $h_{k; \, ...}$ and compare the coefficient of $h_{k; \, e_{j_1}+\dots +e_{j_m}-m+1}$. For $m < l$ we have: 
\[
\begin{split}
\op{Coeff} \left(  h_{k; \, e_{j_1}+\dots +e_{j_m}-m+1},{\rm LHS} \right)  & = 0
\\
\op{Coeff} \left( h_{k; \, e_{j_1}+\dots +e_{j_m}-m+1},{\rm RHS} \right)  
& = \sum_{ \{ j_1,\dots,j_m \} \subset J \subset \{ j_1,\dots,j_l \} } (-1)^{\op{codim}(p) + \#(J) -1} \cdot (-1)^{m-1}
\\
& = \sum_{r=m}^l \binom{l-m}{r-m} (-1)^{ \op{codim}(p)+r-1 } \cdot (-1)^{m-1}
\\
&= \sum_{s=0}^{l-m} \binom{l-m}{s} (-1)^s (-1)^{\op{codim(p)}} = 0
\end{split}
\]
For the second equality, we notice that the sign of the coefficient of $h_{k; \, e_{j_1}+\dots +e_{j_m}-m+1}$ for the RHS only depends on the cardinality of $J$.
Similarly, for $m=l$, we have:
\[
\begin{split}
\op{Coeff} \left( h_{k; \, e_{j_1}+\dots +e_{j_l}-l+1},{\rm LHS}\right)  & = (-1)^{ {\rm codim} (p) }
\\
\op{Coeff} \left( h_{k; \, e_{j_1}+\dots +e_{j_l}-l+1},{\rm RHS}\right)  & = \sum_{s=0}^{0} (-1)^{s} \binom{0}{s} (-1)^{ {\rm codim} (p) } = (-1)^{ {\rm codim} (p) }.
\end{split}
\]
This concludes the demonstration of \eqref{eq:3.1}.
%This proves the lemma.
%\end{proof}
%\begin{lemma}
%\[
%\sum_{ p \in \mathcal{P} ( \underline{a} ) } 
%(-1)^{ {\rm codim} (p) }
%\llangle  \tau_r \tau_s \tau_{ p (\underline{e}) }  \rrangle
%= \llangle  \tau_r \tau_s \tau_{ \underline{e}, \,  \underline{a} }  \rrangle.
%\]
%\end{lemma}
%\begin{proof}

Equaltion \eqref{e:4.3} is a special case of the following more general observation
\[
\sum_{ p \in \mathcal{P} ( \underline{a} ) } 
(-1)^{ {\rm codim} (p) }
\langle  \tau_{r_1} \cdots \tau_{r_n} \tau_{ p (\underline{e}) }  \rangle
= \langle  \tau_{r_1} \cdots \tau_{r_n} \tau_{ \underline{e}; \,  \underline{a} }  \rangle.
\]
This observation follows from the definition of admissible partition and Proposition~\ref{p:1}.

For Equation \eqref{e:4.4} we expand right hand side by equation (\ref{e:1}) and compare the coefficient of the term 
$\langle  \tau_r \tau_{ p ( \underline{e} )_I }  \rangle \langle  \tau_s \tau_{ p (\underline{e})_{I^c} }  \rangle$ on both sides:
\begin{enumerate}
\item[LHS] = $(-1)^{ {\rm codim}(p) }$,
\item[RHS] = $(-1)^{ {\rm codim}(p_I) } (-1)^{ {\rm codim}(p_{I^c}) } $,
\end{enumerate}
where $p_I$ is the partition on $I$ given by the restriction of $p$ on $I$. We define the $p_{I^c}$ in the same way.
Now codim($p$) = codim($p_I$) + codim($p_{I^c}$) gives LHS = RHS.
This proves \eqref{e:4.4}.

The proof of proposition is now complete.
\end{proof}

\subsection{Proof of Lemma~\ref{l:3.5}}

First we notice that it suffices to prove the following three identities:
\[
\begin{split}
[ M_{k_1;a} , M_{k_2;b} ] = 0, 
\quad [ M_{k_1;b}, L_{k_2} ] = (k_1+ \frac{3}{2}) M_{k_1+k_2;b} ,
\quad [L_{k_1}, L_{k_2}] = (k_1-k_2) L_{k_1+k_2}.
\end{split}
\]

First, we prove $[L_{k_1}, L_{k_2}] = (k_1 - k_2) L_{k_1+k_2}$ as follows:

A technical lemma of identities of the function $h_{k; \underline{e}}$ will be used.

\begin{lemma} \label{l:4.3}
For $\underline{e} \in \mathbb{N}^n$, and $\underline{a} \in \A^n$, we have the following equality 
\[
2(k_1-k_2)\  \frac{ h_{k_1+k_2; \underline{e}} }{ \#\op{Aut}(\underline{e}, \underline{a}) }
= \sum_{ [\underline{j}, a_{\underline{j}}] \in \op{Power}(\underline{e}, \underline{a} ) } 
\frac{ \Big[
h_{k_1; \underline{j} } h_{k_2; k_1+ | \underline{j} | - \#(\underline{j}) +1\ \underline{j}^c }  - 
h_{k_2; \underline{j} } h_{k_1; k_2+ |\underline{j} | - \#(\underline{j}) +1\ \underline{j}^c }
\Big] }
{ \#\op{Aut}(\underline{j}, a_{\underline{j}})\ \#\op{Aut}((\underline{j}, a_{\underline{j}})^c )}.
\]
In particular, if $n=1$, $e \in \mathbb{N}$, we have:
\[
2(k_1-k_2) h_{k_1+k_2; e} = h_{k_1;e} h_{k_2; k_1+e} - h_{k_2;e} h_{k_1; k_2+e}.
\]
\end{lemma}
\begin{proof}
The proof for the special ($n=1$) case is easy. Just expand $h_{k;e}$ by definition.

For the general case, we expand the positive part of RHS, called PRHS, by the definition of $h_{k;...}$. For the negative part, it is just interchange $k_1$ and $k_2$. We have RHS = PRHS - NRHS.

Notice that there are exactly three types of terms:
\begin{enumerate}
\item 
$ h_{k_1; |\underline{j} | - \# \underline{j} +1 } 
h_{k_2; k_1 +|\underline{j} | - \# \underline{j} +1} $.
\item 
$ h_{k_1;|\underline{j} | - \# \underline{j} +1 } 
h_{k_2; |\underline{m} | - \# \underline{m} +1} $ where $\underline{m} \cap \underline{j} = \emptyset$.
\item 
$ h_{k_1; |\underline{j} | - \# \underline{j} +1} 
h_{k_2; k_1 + |\underline{j}' | - \# \underline{j}' +1} $ where $\underline{j} \subsetneq \underline{j'}$.
\end{enumerate}
Now we compute the coefficient of the terms of the above three types for PRHS explicitely
\begin{enumerate}
\item 
\[
\begin{split}
&\op{Coeff} \left( h_{k_1; |\underline{j} | - \# \underline{j} +1} 
h_{k_2; k_1 + |\underline{j} | - \# \underline{j} +1},{\rm PRHS} \right)
\\
&= 
\frac{
\op{Coeff} \left( h_{k_1; |\underline{j} | - \# \underline{j} +1 },h_{k_1; \underline{j}}\right)
\cdot
\op{Coeff} \left( h_{k_2; k_1 + |\underline{j} | - \# \underline{j} +1} 
,h_{k_2; k_1+ |\underline{j} | - \# \underline{j} +1, \underline{j}^c} \right)
}
{ \#\op{Aut}  (\underline{j}, a_{\underline{j}}) \#\op{Aut}( (\underline{j}, a_{\underline{j}})^c)  } 
\\
&= 
\frac{(-1)^{\# \underline{j} - 1} \cdot 1 }{ \#\op{Aut}  (\underline{j}, a_{\underline{j}}) \#\op{Aut}( (\underline{j}, a_{\underline{j}})^c)  }
= \frac{(-1)^{\# \underline{j} - 1}}{ \#\op{Aut}  (\underline{j}, a_{\underline{j}}) \#\op{Aut}( (\underline{j}, a_{\underline{j}})^c) }.
\end{split}
\]
\item For $\underline{m} \cap \underline{j} = \emptyset$, we compute 
\[
\begin{split}
& \op{Coeff} \left( h_{k_1; |\underline{j} | - \# \underline{j} +1} 
h_{k_2; |\underline{m} | - \# \underline{m} +1} , {\rm PRHS} \right)
\\
&= \sum_{\underline{j} \subset \underline{l}, \underline{m} \subset \underline{l}^c  }
\frac{
\op{Coeff} \left( h_{k_1; |\underline{j} | - \# \underline{j} +1} , h_{k_1; \underline{l}} \right)
\cdot
\op{Coeff} \left( h_{k_2; |\underline{m} | - \# \underline{m} +1} 
, h_{k_2; k_1+ | \underline{l} | - \#\underline{l} +1\ \underline{l}^c } \right)
}
{
\#\op{Aut}(\underline{l}, a_{\underline{l}}) 
\#\op{Aut}((\underline{l}, a_{\underline{l}})^c)
} 
\\
&= \sum_{\underline{j} \subset \underline{l} \subset \underline{j}'}
\frac{
\Big( \frac{ (-1)^{\#\underline{j}  -1} \op{Aut}  (\underline{l}, a_{\underline{l}} ) }
{ \op{Aut}  (\underline{j}, a_{\underline{j}}) \op{Aut}  (\underline{l} \setminus \underline{j}, a_{\underline{l} \setminus \underline{j}}) } \Big) 
\cdot
\Big( \frac{ (-1)^{\#\underline{m}-1} \op{Aut}  (\underline{l}^c, a_{\underline{l}^c}) }
{ \op{Aut}  (\underline{m}, a_{\underline{m}}) \op{Aut}  (\underline{l}^c \setminus \underline{m}, a_{\underline{l}^c \setminus \underline{m}}) } \Big)
}
{
\#\op{Aut}(\underline{l}, a_{\underline{l}}) 
\#\op{Aut}((\underline{l}, a_{\underline{l}})^c)
} 
\end{split}
\]
We don't need to simply this term since we have 
\[
\op{Coeff} \left(  h_{k_1; |\underline{j} | - \# \underline{j} +1} h_{k_2; |\underline{m} | - \# \underline{m} +1} , {\rm PRHS} \right) =
\op{Coeff} \left( h_{k_1; |\underline{j} | - \# \underline{j} +1} h_{k_2; |\underline{m} | - \# \underline{m} +1} , {\rm NRHS} \right).
\]
\item For $\underline{j} \subsetneq \underline{j}'$, we compute
\[
\begin{split}
&\op{Coeff} \left( h_{k_1; |\underline{j} | - \# \underline{j} +1} h_{k_2; k_1 + |\underline{j} | - \# \underline{j} +1}, {\rm PRHS} \right)
\\
&= \sum_{\underline{j} \subset \underline{l} \subset \underline{j}'}
\frac{
\op{Coeff} \left( h_{k_1; |\underline{j} | - \# \underline{j} +1 },h_{k_1; \underline{l}}] \right)
\cdot
\op{Coeff} \left(
h_{k_2; k_1 + |\underline{j} | - \# \underline{j} +1 } 
,h_{k_2; k_1+ |\underline{l} | - \# \underline{l} +1, \underline{l}^c } \right)
}
{ \#\op{Aut} (\underline{l}, a_{\underline{l}}) 
\#\op{Aut} ((\underline{l}, a_{\underline{l}})^c ) } 
\\
&= \sum_{\underline{j} \subset \underline{l} \subset \underline{j}'}
\frac{
\Big( \frac{ (-1)^{ \#\underline{j}-1 } \#\op{Aut}  (\underline{l}, a_{\underline{l}}) }
{ \#\op{Aut}  (\underline{j}, a_{\underline{j}}) \#\op{Aut}  (\underline{l} \setminus \underline{j}, a_{\underline{l} \setminus \underline{j}})}  \Big)
\cdot
\Big( \frac{ (-1)^{ \#\underline{j}'  - \#\underline{l} -1 } \#\op{Aut}  (\underline{l} \setminus \underline{j}) , a_{\underline{l} \setminus \underline{j}) } }
{ \#\op{Aut}  (\underline{j}' \setminus \underline{l}, a_{\underline{j}' \setminus \underline{l}}) 
\#\op{Aut}  (( \underline{j}', a_{\underline{j}'})^c) } \Big)
}
{ \#\op{Aut} (\underline{l}, a_{\underline{l}}) \#\op{Aut} ((\underline{l}, a_{\underline{l}})^c ) }  
\\
&= \frac{ (-1)^{ \#\underline{j} + \#\underline{j}' } }{ 
\#\op{Aut}  (\underline{j}, a_{\underline{j}}) 
\#\op{Aut} ( (\underline{j}', a_{\underline{j}'})^c ) }
\sum_{\underline{j} \subset \underline{l} \subset \underline{j}'} 
\frac{ (-1)^{ \#\underline{l} } }{ \#\op{Aut}  (\underline{l} \setminus \underline{j}, a_{\underline{l} \setminus \underline{j}}) 
\#\op{Aut}  (\underline{j}'\setminus \underline{l}, a_{\underline{j}'\setminus \underline{l}} ) } =0.
\end{split}
\]
\end{enumerate}
Finally we compute PRHS - NRHS and notice that only type $(1)$ terms will have nontrivial contribution.
\[
\begin{split}
&{ \rm PRHS  -  NRHS } 
\\
&= \sum_{\underline{j}} \frac{(-1)^{ \#\underline{j} -1}}{ \#\op{Aut}(\underline{j}, a_{\underline{j}}) \#\op{Aut}((\underline{j}, a_{\underline{j}})^c)} 
\left(
h_{k_1; |\underline{j} | - \# \underline{j} +1 } h_{k_2; k_1 + |\underline{j} | - \# \underline{j} +1 } -
h_{k_2; |\underline{j} | - \# \underline{j} +1 } h_{k_1; k_2 + |\underline{j} | - \# \underline{j} +1 }
\right)
\\
&= \sum_{\underline{j}} \frac{(-1)^{ \#\underline{j} -1}}{ \#\op{Aut}(\underline{j}, a_{\underline{j}}) \#\op{Aut}((\underline{j}, a_{\underline{j}})^c)} 
2(k_1-k_2) h_{k_1+k_2; |\underline{j} | - \# \underline{j} +1} 
\\
&= \frac{ h_{k_1+k_2; \underline{i}} }{ \#\op{Aut} (\underline{i}, a_{\underline{i}}) }.
\end{split}
\]
\end{proof}
Now we compare the coefficient of $\displaystyle t_{\underline{i}; \underline{a}} \frac{\partial}{\partial t_{\underline{i}; \underline{a}} } $ for $L_{k_1+k_2}$ and $[L_{k_1}, L_{k_2}]$:
\[
\begin{split}
\op{Coeff} \left(
t_{\underline{i}; \underline{a}} \frac{\partial}{\partial t_{\underline{i}; \underline{a}} } ,L_{k_1+k_2}\right) &= \frac{h_{k_1+k_2 ; \underline{i}}}{ \#\op{Aut}(\underline{i}, \underline{a}) }
\\
\op{Coeff} \left(
t_{\underline{i}; \underline{a}} \frac{\partial}{\partial t_{\underline{i}; \underline{a}} } , [ L_{k_1}, L_{k_2} ] \right)
&= \sum_{ [\underline{j}] \in \op{Power}(\underline{e}) } 
\frac{ \Big[
h_{k_1; \underline{j} } h_{k_2; k_1+ | \underline{j} | - \#(\underline{j}) +1\ \underline{j}^c }  - 
h_{k_2; \underline{j} } h_{k_1; k_2+ |\underline{j} | - \#(\underline{j}) +1\ \underline{j}^c }
\Big] }
{\#\op{Aut}(\underline{j}, a_{\underline{j}})\ \#\op{Aut}((\underline{j}, a_{\underline{j}})^c) }.
\end{split}
\]
The equality of the above coefficients is exactly the statement of Lemma ~(\ref{l:4.3}).

We furthermore notice that the leading term
\[
- (2k+3)!! \frac{\partial}{\partial t_{k+1}},
\]
and the quadratic term
\[
\frac{1}{2}\sum_{r+s = k-1} (2r+1)!! (2s+1)!! \frac{\partial}{\partial t_r} \frac{\partial }{\partial t_s}
\]
only depend to weight $1$ variable. Hence checking the virasoro relation for these two terms is the same as that in the unweighted case and is omitted. Now combining the results for weighted, leading and quadratic terms, we get the desired relation:
\[
(k_1 - k_2) L_{k_1+k_2} = [ L_{k_1}, L_{k_2} ].
\]

We now check $ [ M_{k_1; b_1} , M_{k_2;b_2} ] = 0 $.
For any $\underline{e} \in \mathbb{N}^n$ and $\underline{a} \in {\A}^n$ satisfies $b_1+b_2+ |\underline{a}| \leq 1$, we compute
\[
\begin{split}
&\op{Coeff} \left(
t_{ \underline{e}; \underline{a} } 
\frac{\partial}{\partial t_{ k_1+k_2 + 1 + |\underline{e}| -\#(\underline{e}); |\underline{a}| + b_1 + b_2 }}
, M_{k_1;b_1} ( M_{k_2;b_2} ) \right)
\\
&= 
\frac{(2k_1+3)!! (2k_2+3)!! }{4}   
\left(
\frac{1}{\#\op{Aut}(\underline{e}, \underline{a})} + \sum_{\phi \neq \underline{j} \subset {\rm Power}(\underline{e})}
\frac{1}{ \#\op{Aut} (\underline{j}, a_{\underline{j}}) \#\op{Aut}((\underline{j}, a_{\underline{j}})^c) }
\right).
\end{split}
\]
Notice that it is symmetric with respect to $(k_1,b_1)$ and $(k_2, b_2)$. Hence 
\[
\op{Coeff} \left(
t_{ \underline{e}; \underline{a} } 
\frac{\partial}{\partial t_{ k_1+k_2 + |\underline{e}| -\#(\underline{e}) +1 }}
, [M_{k_1;b_1} , M_{k_2;b_2} ] \right) = 0.
\]

We will need the following ``inductive definition'' of the $h$-function in the proof below.

\begin{lemma} \label{2:h}
\begin{equation} \label{ce:h}
h_{k;c, \underline{e}} = h_{k;c} + \sum_{\phi \neq J \subset \{ 1,\dots,n \} }
\Big(
h_{k; \underline{e}_J} - h_{k; c+|\underline{e}_J| - \#(J), \underline{e}_{J^c}}
\Big)
\end{equation}
\end{lemma}

\begin{proof}
We expand both sides by the definition of $h_{k; \dots}$. The proof of the equality reduce to the equality coefficient of the following three types of terms on both sides:
\[
h_{k;c}, \qquad
h_{k; |\underline{e}_I | - \#(I) + 1}, \quad h_{k; c+  |\underline{e}_I | - \#(I) }, \mbox{ where $ \phi \neq  I \subset \{ 1,\dots,n \}$. }
\]
\begin{enumerate}
\item $\op{Coeff} \left(  h_{k;c}, {\rm LHS} \right) = \op{Coeff} \left( h_{k;c}, {\rm RHS} \right)=1$ is clear.

\item We compare the coefficient of  $h_{k; |\underline{e}_I | - \#(I) + 1}$ on both sides as follows:
\[
\begin{split}
&\op{Coeff} \left(
h_{k; |\underline{e}_I | - \#(I) + 1},{\rm LHS}\right) = (-1)^{ \#(I)-1 };
\\
&\op{Coeff} \left( h_{k; |\underline{e}_I | - \#(I) + 1},{\rm RHS}\right)
\\
&=
 \op{Coeff} \left(
h_{k; |\underline{e}_I | - \#(I) + 1}, \sum_{\phi \neq J \subset \{ 1,\dots,n \} }
\Big(
h_{k; \underline{e}_J} - h_{k; c+|\underline{e}_J| - \#(J), \underline{e}_{J^c}}
\Big) \right)
\\
&= \op{Coeff} \left( h_{k; |\underline{e}_I | - \#(I) + 1}, \sum_{ I \subset J \subset \{ 1,\dots,n \} } h_{k;\underline{e}_J} \right)
\\
&\qquad\qquad\qquad - \op{Coeff} \left(  h_{k; |\underline{e}_I | - \#(I) + 1},\sum_{\phi \neq J \subset \{1,\dots,n\}, J \cap I = \phi } h_{k;\underline{e}_J} \right)
\\
&= (-1)^{\#(I)-1} 2^{n-\#(I)} - (-1)^{\#(I)-1}( 2^{n-\#(I)} -1 ) = (-1)^{\#(I)-1}.
\end{split}
\]
This proves the equality of the coefficients of the second type.
\item We compare the coefficient of $h_{k; c+  |\underline{e}_I | - \#(I) }$ on both sides as follows:
\[
\begin{split}
&\op{Coeff} \left(
h_{k; c+  |\underline{e}_I | - \#(I) },{\rm LHS}\right) = (-1)^{\#(I)};
\\
& \op{Coeff} \left( h_{k; c+  |\underline{e}_I | - \#(I) },{\rm RHS}\right) 
\\
&=
\op{Coeff} \left( -h_{k; c+  |\underline{e}_I | - \#(I) }, \sum_{\phi\neq J \subset I} h_{k; c+|\underline{e}_J| - \#(J), \underline{e}_{J^c}} \right)
\\
&= -\sum_{\phi \neq J \subset I} (-1)^{\#(I) - \#(J)} = (-1)^{\#(I)}.
\end{split}
\]
This proves the equality of the coefficients of the third type.
\end{enumerate}
\end{proof}

Finally, we prove $[ M_{k_1;b}, L_{k_2} ] = (k_1+ \frac{3}{2}) M_{k_1+k_2;b}$.

\begin{lemma} 
For any $k_1,k_2 \geq -1$, $\underline{e} \in \mathbb{N}^n$, and $a \in \A^n$, we have the following identity for $h$-function:
\[
-   \frac{  h_{ k_2 ; k_1+1, \underline{e} }  }{ \#\op{Aut}(\underline{e}, \underline{a}) }   
+ \sum_{\phi \neq [\underline{j}, a_{\underline{j}}] \subset {\rm Power}(\underline{e}, \underline{a}) } 
\frac{ - h_{k_2; k_1+1 + |\underline{j}| -\#(\underline{j}) , \underline{j}^c}
+  h_{k_2; \underline{j}} }{\#\op{Aut}(\underline{j}, a_{\underline{j}}) \#\op{Aut}((\underline{j}, a_{\underline{j}})^c )}
= - \frac{ (2k_1+2k_2+3)!! }{ (2k_1+1)!! \#\op{Aut}(\underline{e}, \underline{a})}.
\]
Moreover, the identity is equivalent to the relation: $[ M_{k_1;b}, L_{k_2} ] = ( k_1 + \frac{3}{2} )M_{k_1+k_2;b}$.
\end{lemma}
\begin{proof}
For the second statement, notice that 
\[
\begin{split}
&\op{Coeff} \left(
t_{\underline{e};\underline{a}}  
\frac{\partial}{\partial t_{k_1+k_1+1 + \underline{e} - \#(\underline{e}); b+|\underline{a}|  }}
, 2M_{k_1;b} (2L_{k_2}) \right)
\\
& \qquad = - (2k_1+3)!! 
\left(
\frac{  h_{ k_2 ; k_1+1, \underline{e} }  }{ \#\op{Aut}(\underline{e}, \underline{a}) }  
+ \sum_{\phi \neq [\underline{j}, a_{\underline{j}}] \subset {\rm Power}(\underline{e}, \underline{a}) } 
\frac{h_{k_2; k_1+1 + |\underline{j}| -\#(\underline{j}) , \underline{j}^c}}{ \#\op{Aut}(\underline{j}, a_{\underline{j}}) \#\op{Aut}((\underline{j}, a_{\underline{j}})^c) } 
\right), 
\end{split} 
\]
and that
\[
\begin{split}
&\op{Coeff} \left(
t_{\underline{e};\underline{a}} 
\frac{\partial}{\partial t_{k_1+k_1+1 + \underline{e} - \#(\underline{e}); b+|\underline{a}|  }}
,  2L_{k_2} (2M_{k_1;b}) \right)
\\
& \qquad\qquad= - (2k_1+3)!! \sum_{\phi \neq [\underline{j}, a_{\underline{j}}] \subset {\rm Power}(\underline{e}, \underline{a}) } 
\frac{ h_{k_2;\underline{j}} }{ \#\op{Aut}(\underline{j}, a_{\underline{j}}) \#\op{Aut}((\underline{j}, a_{\underline{j}})^c) }.
\end{split}
\]
The identity of the coefficient of $\displaystyle t_{\underline{e};\underline{a}} 
\frac{\partial}{\partial t_{k_1+k_1+1 + \underline{e} - \#(\underline{e}); b+|\underline{a}|  }}$ of the equation 
$[ M_{k_1;b}, L_{k_2} ] = ( k_1 + \frac{3}{2} )M_{k_1+k_2;b}$ gives the identity for $h$-function described in the lemma.

To prove the first statement, we expand $h_{k_2,...}$ by definition. It suffices to prove the following equivalence relation:
\[
\begin{split}
-h_{k_2;k_1+1,\underline{e}} + \sum_{\phi \neq J \subset \{1,\dots,n\} } 
h_{k_2;k_1+1 + |\underline{e}_J| -\#(J), \underline{e}_{J^c} }
&+ \sum_{\phi \neq J \subset \{1,\dots,n\} } h_{k_2; \underline{e}_J}
\\
& = -h_{k_2;k_1+1} = - \frac{(2k_1+2k_2+3)!!}{(2k_1+1)!!},
\end{split}
\]
which is exactly the inductive definition for $h$-function introduced in Lemma~\ref{2:h}. 
This completes the proof of lemma.
\end{proof}

\section{From Virasoro to KdV hierarchy} \label{s:KdV}
We recall the Gelfand--Dickey formulation of the KdV hierarchy. 
A formal power series $U(t_0,t_1,\dots)$ is said to satisfy the KdV hierarchy if it satisfies the following systems of differential equations
\begin{equation} \label{e:4.1}
\frac{\partial U}{\partial t_i } = \frac{\partial}{\partial t_0} R_i[U], \qquad i \geq 0,
\end{equation}
where $R_i[U]$ are polynomials in $U$ and its derivatives with respect to $t_0$ and are defined recursively by
\begin{equation} \label{e:4.2}
R_0 = U, \qquad \frac{\partial R_{n+1}}{\partial t_0} = 
\frac{1}{2n+1} \Big(
\frac{\partial U}{\partial t_0} + 2 U \frac{\partial}{\partial t_0}
+ \frac{1}{4} \frac{\partial^3}{\partial t_0^3}
\Big) R_n.
\end{equation}
%\begin{theorem} \label{t:5}
Let 
\[
 \llangle  t_{i_1} \dots t_{i_r}  \rrangle := \frac{\partial^r F}{\partial t_{i_1} \dots \partial t_{i_r}}, \qquad
%\]
%\[ \displaystyle 
U = \frac{\partial^2 F}{\partial t_0^2} 
%= \frac{\partial^2}{\partial t_0^2}  \left( \sum_{n} \frac{1}{n!} \sum_{ \underline{k} \in \mathbb{N}^n } t_{k_1} \dots t_{k_n} \langle \tau_{ \underline{k}} \rangle \right)
= \llangle t_0 t_0  \rrangle,
\] 
where $F(t)$ is the generating function for weight $1$ correlators.
Then Witten's conjecture on intersection numbers on $\Mbar_{g,n}$ is equivalent to the statement that \emph{$U$ satisfies the KdV hierarchy} \cite{eW}. %formulated above. 
%\end{theorem}
This can also be formulated in the following form
%\begin{corollary}
%With the same $\llangle  t_{i_1} \dots t_{i_r}   \rrangle $ notation as before, we have
\[
\llangle t_0t_0 \rrangle \llangle t_{n+1} \rrangle = \frac{1}{2n+1}
\Big(
\llangle t_0t_0t_0 \rrangle \llangle t_0 t_n \rrangle
+ 2 \llangle t_0t_0 \rrangle \llangle t_0t_0t_n \rrangle
+ \frac{1}{4} \llangle t_0t_0t_0t_0 t_n \rrangle
\Big).
\]
%\end{corollary}

The weighted case can be treated similarly by replacing $t_k$ with $\mathbf{t}_{k;a}$ throughout.
In other words, in the Hamiltonian formulation of the KdV hierarchy, $\mathbf{t}_{k;a}$ serve as the formal time variables.

Now define $R_{n;b}[U]$, the polynomials in $U$ and its derivatives with respect to $\mathbf{t}_{0;b}$, inductively as follows
\[
R_{0;b} = U , \qquad \frac{\partial R_{n+1;b}}{\partial \mathbf{t}_{0;b}} = 
\frac{1}{2n+1} \Big(
\frac{\partial U}{\partial \mathbf{t}_{0;b}} + 2 U \frac{\partial}{\partial \mathbf{t}_{0;b}}
+ \frac{1}{4} \frac{\partial^3}{\partial \mathbf{t}_{0;b}^3}
\Big) R_{n;b}.
\]

%Now we have the weighted version for Theorem ~\ref{t:5}.
\begin{theorem} \label{t:4.1}
Let $\displaystyle U^{\A} = \frac{\partial^2 F^{\A}}{\partial \mathbf{t}_{0;b}^2} 
= \frac{\partial^2}{\partial \mathbf{t}_{0;b}^2}
\Big( \sum_{n} \frac{1}{n!} \sum_{ ( \underline{k}, \underline{a} ) \in \mathbb{N}^n \times \A^n }
t_{\underline{k}; \underline{a}} \langle  \tau_{\underline{k}; \, \underline{a}} \rangle
\Big)$. Then $U^{\A}$ satisfies the KdV hierarchy for any $b \in \A$.
More concretely, $U^{\A}$ satisfies the following system of differential equations
\[
\frac{\partial U^{\A}}{\partial \mathbf{t}_{i;b}} = \frac{\partial}{\partial \mathbf{t}_{0;b}} R_{i,b}[U^{\A}], 
\quad \mbox{ for $i \geq 0$, $b \in \A$. }
\]
\end{theorem}

\begin{proof}
We play the same trick as in the proof of Theorem~\ref{t:2.13}. 

For the case $1 \in \A$, Notice that differentiate $F^{\A}$ with respect to $t_i$ and $\mathbf{t}_{i;b}$ are identical. Hence $U^{\A}$ satisfies the KdV hierarchy with respect to the formal time variables $\mathbf{t}_{k;a}$.

In case $1 \notin \A$, one can first extend the setup to $\bar{\A} = \A \cup \{1 \}$ and then restrict to the subspace $\{t_{i;1}=0\}_{i \in \mathbb{N}}$. 
Since $\partial_{\mathbf{t}_{k;b}}$ does not involve the derivative with respect to weight $1$ variables, the restriction respsects all the above arguments.
This completes the proof.
\end{proof}

%\begin{remark}
Some remarks are in order.
The original Witten's conjecture is stated in two parts \cite{eW}.
The first part of the conjecture states that $U(t) := \partial^2 F(t) / \partial t_0^2$ satisfies the KdV hierarchy;
the second part states that $F(t)$ obeys the string equation $L_{-1} F(t) =0$.
There, the \emph{initial condition} $U(t) |_{t_{\geq 1} =0} = t_0$ and the equation \eqref{e:4.1} and \eqref{e:4.2} from the KdV hierarchy uniquely determine $U(t)$ in the obvious way. %and with more efforts $F(t)$.
In the weighted case, the ``string equation'' $L^{\A}_{-1;a} F^{\A} =0$ appears in a modified form.
Theorem~\ref{t:4.1} confirms that $U^{\A}$ also satisfies the KdV hierarchy, with the same set of equations in a different coordinate system $\{ \mathbf{t}_{k;a} \}$.
One can ask what the corresponding initial condition $U^{\A} (\mathbf{t}) |_{\mathbf{t}_{\geq 1; b} =0}$ is in this new coordinate system.
(This was also brought up in a correspondence with Youjin Zhang.) Our answer is given below.

We deinfe variables $\{ \mathbf{t}_{k;b} \}_{k \geq 0,\ b \in \mathcal{A}}$ as follows:
\[
\mathbf{t}_{k;b} := 
\sum_{m \geq 1} 
\frac{(-1)^{m-1}}{m!} 
\sum_{\substack
{(\underline{e}, \underline{a}) \in (\mathbb{N} \times \mathcal{A})^m
\\
|\underline{a}|=b,\ | \underline{e}|-m+1 = k
}
}
t_{ \underline{e}; \underline{a} }.
\]
This change of varieables is invertible:
\begin{lemma}
With $\{ \mathbf{t}_{k;b} \}_{k \geq 0,\ b \in \mathcal{A}}$ defined above, we have
\[
t_{k;b} = \sum_{n \geq 1} \frac{1}{n!}  
\sum_{\substack
{(\underline{e}; \underline{a}) \in (\mathbb{N} \times \mathcal{A})^n
\\
|\underline{a}|=b,\ | \underline{e}|-n+1 = k
}
}
\mathbf{t}_{ \underline{e}; \underline{a} }.
\]
\end{lemma}
Moreover $\{ \mathbf{t}_{k,b} \}$ forms a  coordinate system corresponds to vector fields $\{ \frac{\partial}{\partial \mathbf{t}_{k,b}} \}$. Explicitly, we have
\begin{lemma}
With $\{ \mathbf{t}_{k;b} \}_{k \geq 0,\ b \in \mathcal{A}}$ defined above, we have
\[
\frac{\partial}{\partial \mathbf{t}_{k_1; b_1}}( \mathbf{t}_{k_2;b_2} ) =
\begin{cases}
1, \mbox{ if $(k_1,b_1) = (k_2,b_2)$};
\\
0, \mbox{ otherwise}.
\end{cases}
\]
\end{lemma}
\begin{proof}
First we notice that $\frac{\partial}{\partial \mathbf{t}_{k_1; b_1}}( \mathbf{t}_{k_2;b_2} )$ will be a monomial with weight $b_2-b_1$.
Now we divide the computation in three cases:
\begin{enumerate}
\item If $b_1 > b_2$, then $\frac{\partial}{\partial \mathbf{t}_{k_1; b_1}}( \mathbf{t}_{k_2;b_2} )$ should be zero since we don't have negative weight monomial.
\item If $b_1=b_2=b$, then
\[
\frac{\partial}{\partial \mathbf{t}_{k_1; b}}( \mathbf{t}_{k_2;b} ) = \frac{\partial}{\partial t_{k_1;b}} (t_{k_2;b}) = \delta_{k_1,k_2}.
\]
\item If $b_1<b_2$, given $(\underline{e}; \underline{a}) \in (\mathbb{N} \times \mathcal{A})^n$ with $|\underline{a}| = b_2-b_1$ and $\underline{e} = \# (\underline{e}) + k_2-k_1$, then
\[
{\rm Coeff}\Big(
t_{\underline{e} ; \underline{a}}, \frac{\partial}{\partial \mathbf{t}_{k_1;b_1} }(\mathbf{t}_{k_2;b_2}) 
\Big)
= \sum_{I \subset \{ 1,2,\dots,n \}} 
\frac{1}{ {\rm Aut} (t_{\underline{e}_I; \underline{a}_I}) }
\frac{ (-1)^{\#(I^c)} }{ {\rm Aut} (t_{\underline{e}_{I^c}; \underline{a}_{I^c}}) }
=0.
\]
\end{enumerate}
Combine all three cases, we prove the lemma.
\end{proof}
\begin{corollary}
Let $\mathbf{t}_k := \sum_{a\in \mathcal{A}} \mathbf{t}_{k;a}$. Then we have
\[
U^{\mathcal{A}} = U^{\mathcal{A}}( \{ t_{k;a} \}_{k\in \mathbb{N}, a \in \mathcal{A}} ) = U^{1}(\mathbf{t}_0,\mathbf{t}_1,\dots),
\]
where $ \displaystyle U^1(t_0,t_1,\dots) = \frac{\partial^2}{\partial t_0^2} \Big( \sum_{n \geq 1} \frac{1}{n!} \sum_{\underline{k} \in \mathbb{N}^n} t_{\underline{k}} \langle \tau_{\underline{k}} \rangle \Big) $ is the potential corresponds to the unweighted case.
\end{corollary}
\begin{proof}
With the above change of varieables, we can view $U^{\mathcal{A}}$ as formal power series of $\{ \mathbf{t}_{k;a} \}_{k\in \mathbb{N}, a \in \mathcal{A}}$. Then we notice that both side satisfies the KdV hierarchy described in Theorem \ref{t:4.1} and have the same initial condition:
\[
U^{\mathcal{A}}\bigl\vert_{\mathbf{t}_{i>0;a}=0} = U^1(\mathbf{t}_0,\dots) \bigl\vert_{\mathbf{t}_{i>0}=0} = \mathbf{t}_0 .
\]
\end{proof}

\end{document}